\documentclass[reqno]{amsart}

\usepackage{
    enumitem, %
    amsmath, %
    amssymb, %
    bm,
    mathdots,
}

\usepackage{amsthm}
\theoremstyle{plain}
\newtheorem{theorem}{Theorem}[section]
\newtheorem{proposition}[theorem]{Proposition}
\newtheorem{lemma}[theorem]{Lemma}

\theoremstyle{remark}
\newtheorem{remark}[theorem]{Remark}

\numberwithin{equation}{section}

\usepackage[pdftex, colorlinks=true]{hyperref}

\title[Inhomogeneity of Isoparametric Hypersurfaces of OT-FKM-type]{Inhomogeneity of Isoparametric Hypersurfaces of OT-FKM-type in the Pseudo-Sphere}
\author[Y. Sasahara]{Yuta Sasahara}
\thanks{This work was supported by JST SPRING, Grant Number \textsf{JPMJSP2156}.}
\address{Department of Mathematical Sciences, Graduate School of Science, Tokyo Metropolitan University, Minami-Osawa 1-1, Hachioji, Tokyo, 192-0397, Japan}
\email{sasahara-yuta@ed.tmu.ac.jp, ysasahara0804+res@gmail.com}

\begin{document}

\begin{abstract}
    We study isoparametric hypersurfaces, whose principal curvatures are all constant, in the pseudo-Riemannian space forms.
    In this paper, we investigate two topics.
    Firstly, according to representations of Clifford algebras, we give a construction of Clifford systems of signature $(m, r)$ for any $(m, r)$ explicitly.
    Secondly, we show that a (connected) isoparametric hypersurface of OT-FKM-type whose focal variety is $M_{+}$ in the pseudo-sphere is inhomogeneous if the signature $(m, r)$ of its Clifford system on $\mathbb{R}^{2l}_{s}$ satisfies $m\equiv 0\pmod{4}$, $r\equiv 0\pmod{2}$ and $l>m$, showing that each connected component of $M_{+}$ is inhomogeneous.
\end{abstract}

\maketitle

\section{Introduction}
The study of isoparametric hypersurfaces in the pseudo-Riemannian space forms has been started in the 1980s.
In \cite{MR728336}, Nomizu introduced the notion of isoparametric for space-like hypersurfaces in the Lorentzian space forms.
In \cite{MR753432}, Hahn gave the definition of isoparametric hypersurfaces which was generalized from Nomizu's definition: a non-degenerate hypersurface in the pseudo-Riemannian space form is isoparametric if its principal curvatures (the eigenvalue functions of its shape operator) are all constant.
In general, the study of hypersurfaces in the pseudo-Riemannian space forms is much more difficult than that in the Riemannian space forms since their shape operator is not necessarily real-diagonalizable.
Their principal curvatures may be complex numbers.

After \cite{MR753432}, classification problems of Lorentzian isoparametric hypersurfaces have been developed in \cite{MR783023}, \cite{MR1696128} and \cite{MR3077208}.
In these papers, the authors adopted a definition which was different from Hahn's definition: a non-degenerate hypersurface in the pseudo-Riemannian space forms is isoparametric if its shape operator has the constant minimal polynomial, which was first introduced by Magid.
Hahn's former definition follows from Magid's latter definition.
However, the hypersurface in \cite[2.5]{MR753432} is isoparametric as Hahn's definition and is not isoparametric as Magid's definition.

The purpose of this paper is to investigate two topics of isoparametric hypersurfaces in the pseudo-Riemannian space forms.

Firstly, we give a construction of Clifford systems of signature $(m, r)$ for any $(m, r)$, where $m$ is an integer larger than $1$ and $r\in \{0, 1, 2, \ldots, m\}$, in Section 3.
Ozeki and Takeuchi gave inhomogeneous examples of isoparametric hypersurfaces in the sphere in \cite{MR454888}-\cite{MR454889}.
Ferus, Karcher and M\"unzner generalized the construction in \cite{MR624227}. 
In \cite{MR753432}, Hahn extended the construction to that in the pseudo-sphere and gave a great number of examples of such hypersurfaces.
Actually, these hypersurfaces are also isoparametric as Magid's definition since their shape operators are real-diagonalizable or complex-diagonalizable by \cite[2.7.(3), 3.4.(6), 3.5.(1), 3.6, 3.7]{MR753432}.
A hypersurface obtained by using the construction is called an isoparametric hypersurface of OT-FKM-type.
We need to use a Clifford system to construct isoparametric hypersurfaces of OT-FKM-type.
We can obtain Clifford systems of signature $(m, r)$ by using representations of Clifford algebras.

Secondly, we show that a (connected) isoparametric hypersurface of OT-FKM-type whose focal variety is $M_{+}$ in the pseudo-sphere is inhomogeneous if the signature $(m, r)$ of its Clifford system on $\mathbb{R}^{2l}_{s}$ satisfies $m\equiv 0\pmod{4}$, $r\equiv 0\pmod{2}$ and $l>m$, showing that each connected component of $M_{+}$ is inhomogeneous, in Section 4.
In Riemannian case, an isoparametric hypersurface of OT-FKM-type is inhomogeneous in infinitely many cases.
We can use the classification of homogeneous hypersurfaces in the sphere to investigate it.
In pseudo-Riemannian case, we do not know clearly whether an isoparametric hypersurface of OT-FKM-type is homogeneous or not, since we do not know the classification of homogeneous hypersurfaces in the pseudo-sphere.
By the method in \cite[5.8]{MR624227}, we provide a geometric proof of inhomogeneity in that case.
In this paper, we focus on $M_{+}$, not on $M_{-}$.

And lastly, in Section 5, we investigate the diffeomorphism type of inhomogeneous isoparametric hypersurfaces of OT-FKM-type whose Clifford systems are obtained in Section 3.

The author is grateful to Prof. Takashi Sakai for helpful discussions.

\section{Preliminaries}
In this section, we prepare our notations and state basic facts of isoparametric hypersurfaces of OT-FKM-type.
Throughout this paper, we suppose that all of manifolds and mappings are smooth unless otherwise noted.
We use notations in Table \ref{tab:c-n}.
\begin{table}[h]
\caption{conventional notations}
\label{tab:c-n}
\centering
\begin{tabular}{ccl}
    $\delta_{ij}$&:& the Kronecker delta\\
    ${}^{t}\!A$&:& the transpose of a matrix $A$\\
    $E_{n}$&:& the identity matrix of order $n$, where $n$ is a positive integer\\
    $A\oplus B$&:& the direct sum of matrices $A$ and $B$\\
    $A\otimes B$&:& the Kronecker product of matrices $A$ and $B$\\
    $M\cong M'$&:& a manifold $M$ is diffeomorphic to a manifold $M'$\\
    $J_{i, j}:=(-E_{i})\oplus E_{j}$&:& $J_{0, j}$ and $J_{i, 0}$ mean $E_{j}$ and $-E_{i}$, respectively.
\end{tabular}
\end{table}
\subsection{The Pseudo-Riemannian Space Forms and Submanifolds}
Let $n$ be an integer larger than 1 and $s \in \{0, 1, \ldots, n\}$.
Let $\mathbb{R}^{n}_{s}$ be an $n$-dimensional pseudo-Euclidean space provided with the pseudo-inner product $\langle u, v\rangle={}^{t}\! uJ_{s, n-s}v$ for $u, v\in \mathbb{R}^{n}_{s}$.
We define the pseudo-sphere $S^{n}_{s}:=\left\{x\in \mathbb{R}^{n+1}_{s}\ \middle|\ \langle x,x\rangle=1\right\}$ and the pseudo-hyperbolic space $H^{n}_{s}:=\left\{x\in \mathbb{R}^{n+1}_{s+1}\ \middle|\ \langle x,x\rangle=-1\right\}$ of index $s$.
The spaces $S^{n}_{s}$, $\mathbb{R}^{n}_{s}$ and $H^{n}_{s}$ are called the pseudo-Riemannian space forms, whose sectional curvatures of any non-degenerate tangent plane at each point are constantly $\kappa=1, 0$ and $-1$, respectively.
The $n$-dimensional pseudo-Riemannian space form of index $s$ with constant sectional curvature $\kappa$ is denoted by $N^{n}_{s}(\kappa)$.
Let $\bar{\nabla}$ be the natural Levi-Civita connection of $N(\kappa)$.

We fix $x\in N(\kappa)$.
Let $v\in T_{x}N(\kappa)$.
We set $\tau:=\kappa\cdot \langle v, v\rangle$.
The curve $\gamma_{x, v}\colon \mathbb{R}\to N(\kappa)$ is given by 
\begin{equation*}
    \gamma_{x, v}(t)
    :=
    \begin{cases}
        \cos{\left(\sqrt{\tau} t\right)}x
        +
        \left(\sqrt{\tau}\right)^{-1}
        \sin{\left(\sqrt{\tau} t\right)}v
        ,
        &
        (\tau>0),
        \\
        x+tv
        ,
        &
        (\tau=0),
        \\
        \cosh{\left(\sqrt{-\tau} t\right)}x
        +
        \left(\sqrt{-\tau}\right)^{-1}
        \sinh{\left(\sqrt{-\tau} t\right)}v
        ,
        &
        (\tau<0).
    \end{cases}
\end{equation*}
Then $\gamma_{x, v}$ satisfies $\gamma_{x, v}(0)=x$, $\gamma_{x, v}'(0)=v$ and $\bar{\nabla}_{\gamma_{x, v}'}\gamma_{x, v}'=0$.
Thus, $\gamma_{x, v}$ is a geodesic in $N(\kappa)$.
Hence, $N(\kappa)$ is geodesically complete.
We usually write $\exp_{x}(tv)=\gamma_{x, v}(t)$.

Let $n$ be a positive integer and $s \in \{0, 1, \ldots, n+1\}$. Let $M\subset N^{n+1}_{s}(\kappa)$ be a non-degenerate submanifold.
The induced metric of $M$ is also denoted by $\langle \cdot , \cdot \rangle$.
Let $\iota\colon M\hookrightarrow N(\kappa)$ be the inclusion map. 
We then have the orthogonal decomposition
$
    \iota^{*}TN(\kappa)=TM\oplus T^{\bot}M 
$
as vector bundles.
Let 
\begin{equation*}
    \mathrm{Sym}(M):=\left\{F\in \mathrm{End}(TM) \ \middle|\ \forall X, Y\in \Gamma(TM),\ \langle FX, Y\rangle=\langle X, FY\rangle\right\}
    .
\end{equation*}
If the index of $M$ is $d$, a matrix representation $Q$ of $F\in \mathrm{Sym}(M)$ satisfies ${}^{t}\!Q=J_{d, \dim{M}-d}QJ_{d, \dim{M}-d}$.
Let $\eta$ be a normal vector field on $M$.
We define the shape operator $S_{\eta}\in \mathrm{End}(TM)$ associated to $\eta$ given by $S_{\eta}(X)$ is the tangent component of $-\bar{\nabla}_{X}\eta\in \Gamma(\iota^{*}TN(\kappa))$ for $X\in \Gamma(TM)$.
Then we have $S_{\eta}\in \mathrm{Sym}(M)$.

\subsection{Isoparametric Hypersurfaces}
From now on, let $M^{n}\subset N(\kappa)^{n+1}_{s}$ be a non-degenerate hypersurface and suppose that there is a unit normal vector field $\xi$ on $M$.
Let $\nu:=\langle \xi, \xi\rangle$ and $\delta:=\kappa\cdot \nu\in \{-1, 0, 1\}$.
We call $\delta$ the type of $M$.
If $\nu=1$, the index of $M$ is $s$ and then $s\leqq n$ is necessary.
If $\nu=-1$, the index of $M$ is $s-1$ and then $s\geqq 1$ is necessary.

The connected hypersurface $M$ is called \textit{isoparametric} if its principal curvatures (the eigenvalue functions of $S_{\xi}$) are all constant.
Let $M\subset N^{n+1}_{s}(-1)=H^{n+1}_{s}$ be an isoparametric hypersurface of type $\delta$.
Referring to \cite[2.1]{MR753432}, by replacing the metric $\langle \cdot, \cdot\rangle$ of $N^{n+1}_{s}(-1)$ with $-\langle\cdot, \cdot\rangle$, we obtain the isoparametric hypersurface $(M, -\langle\cdot, \cdot\rangle)\subset N^{n+1}_{n+1-s}(1)=S^{n+1}_{n+1-s}$ of type $\delta$.
Thus, it is sufficient that we study isoparametric hypersurfaces in the pseudo-Riemannian space forms of constant curvature $\kappa=0, 1$.
In this paper, we consider the case of $\kappa=1$ mainly.

A function $f\colon N^{n+1}_{s}(\kappa)\to \mathbb{R}$ is isoparametric if there exist functions $\Phi, \Psi\colon \mathbb{R}\to \mathbb{R}$ such that $\langle \mathrm{grad}^{N}(f), \mathrm{grad}^{N}(f)\rangle=\Phi\circ f$ and $\Delta^{N}(f)=\Psi\circ f$ hold, where $\mathrm{grad}^{N}(f)$ and $\Delta^{N}(f)$ are the gradient and the Laplacian of $f$ in $N^{n+1}_{s}(\kappa)$, respectively.
For an isoparametric function $f$, we set 
$
    \mathrm{W}_{\mathrm{RN}}(f):=
    \left\{
        c\in f\left(N(\kappa)^{n+1}_{s}\right)
        \ \middle|\ 
        \Phi(c)\neq 0
    \right\}
$.
By \cite[2.2]{MR753432}, if $f\colon N^{n+1}_{s}(\kappa)\to \mathbb{R}$ is isoparametric and $c\in \mathrm{W}_{\mathrm{RN}}(f)$, then each connected component of $f^{-1}(c)\subset N^{n+1}_{s}(\kappa)$ is an isoparametric hypersurface.

\subsection{Clifford Systems}
\label{sect:C-S}
Let $l$ be an positive integer and $s\in \{0, 1, 2, \ldots, 2l\}$.
The identity transformation of $\mathrm{Sym}(\mathbb{R}^{2l}_{s})$ is denoted by $I$.
Let $m$ be an integer larger than $1$ and $r\in \{0, 1, 2, \ldots, m\}$.
A Clifford system of signature $(m, r)$ on $\mathbb{R}^{2l}_{s}$ is a family $\{P_{i}\}_{i=1}^{m}$ such that $P_{i}\in \mathrm{Sym}(\mathbb{R}^{2l}_{s})\ (\forall i\in \{1, \ldots, m\})$ and
$
    P_{i}P_{j}+P_{j}P_{i}=2\eta_{ij}I
    \ (\forall i, j\in \{1, \ldots, m\})
$
hold where $(\eta_{ij})=J_{r, m-r}$.

We give a pseudo-metric on $\mathrm{Sym}(\mathbb{R}^{2l}_{s})$ by
$\langle Q, Q'\rangle:=(2l)^{-1}\mathrm{trace}(QQ')$ for $Q, Q'\in \mathrm{Sym}(\mathbb{R}^{2l}_{s})$.
Let $\Sigma:=\mathrm{span}\left\{P_{i}\ \middle|\ i\in \{1, \ldots, m\}\right\}$.
Referring to \cite[3.1]{MR753432}, we have some properties of Clifford systems:
\begin{enumerate}[label=(\arabic*)]
    \item The Clifford system $\{P_{i}\}_{i=1}^{m}$ is a pseudo-orthonormal basis of $\Sigma$.
    \item Every pseudo-orthonormal basis of $\Sigma$ is a Clifford system of signature $(m, r)$.
    \item \label{change-PQ}
    For $Q, Q'\in \Sigma$, $QQ'+Q'Q=2\langle Q, Q'\rangle I$
    holds.
    \item \label{metric}
    For $Q, Q'\in \Sigma,\ x\in \mathbb{R}^{2l}_{s}$, 
    $\langle Qx, Q'x\rangle=\langle Q, Q'\rangle\langle x, x\rangle$
    holds.
    \item \label{index-equiv}If $r>0$, then $s=l$.
    \item \label{func-depend}The function $H\colon \mathbb{R}^{2l}_{s}\to \mathbb{R}$ is given by
    \begin{equation*}
        H(x)
        :=
        \sum_{j=1}^{m}\eta_{jj}
        \langle P_{j}x, x\rangle^{2}
        ,
    \end{equation*}
    which depends only on $\Sigma$, not on $\{P_{i}\}_{i=1}^{m}$.
\end{enumerate}

A square matrix $A$ of order $(i+j)$ is a pseudo-orthogonal matrix of signature $(i, j)$ if ${}^{t}\!AJ_{i, j}A=J_{i, j}$ holds.
The set of all pseudo-orthogonal matrices of signature $(i, j)$ is denoted by $O(i, j)$.
For example, $Q\in \Sigma$ satisfying $\langle Q, Q\rangle=1$ is in $O(s, 2l-s)$.
For a pseudo-orthonormal basis $\{Q_{i}\}_{i=1}^{m}$ of $\Sigma$, there exists $A\in O(r, m-r)$ such that 
$
    \begin{bmatrix}
        Q_{1} & Q_{2} & \cdots & Q_{m}
    \end{bmatrix}
    =
    \begin{bmatrix}
        P_{1} & P_{2} & \cdots & P_{m}
    \end{bmatrix}
    A
$.
Let $O(j):=O(0, j)$.

\subsection{Isoparametric Hypersurfaces of OT-FKM-Type}
\label{sect:IHoO}
The homogeneous function $F\colon \mathbb{R}^{2l}_{s}\to \mathbb{R}$ of degree $4$ is given by
$
    F(x)
    :=
    \langle x, x\rangle^{2}
    -
    2
    H(x)
    .
$
Let $f:=\left.F\right|_{S^{2l-1}_{s}}$.
Referring to \cite[3.2]{MR753432}, $f$ is isoparametric.
Let $c\in \mathrm{W}_{\mathrm{RN}}(f)=f(S^{2l-1}_{s})\setminus \{-1, 1\}$ and $M_{c}:=f^{-1}(c)\subset S^{2l-1}_{s}$.
Then each connected component of $M_{c}$ is an isoparametric hypersurface, which is called an isoparametric hypersurface of OT-FKM-type.
In Riemannian case, that is, in the case where $s=0$ and $r=0$, we note that $M_{c}$ is connected.
(We refer to \cite[3.6, Connectedness of the level sets of $F$]{MR3408101}.)
We note that $\kappa=1$ here.
The type of $M_{c}$ is $\delta=1$ if $c\in (-1, 1)$, and $\delta=-1$ if $c\in (-\infty, -1)\cup(1, \infty)$.
Referring to \S\ref{sect:C-S}\ref{func-depend}, $\{M_{c}\}_{c\in \mathrm{W}_{\mathrm{RN}}(f)}$ depends only on $\Sigma$, not on $\{P_{i}\}_{i=1}^{m}$.
We note that $s\leqq 2l-1$.

Let $c\in \mathrm{W}_{\mathrm{RN}}(f)$ and $x\in M_{c}$.
Referring to \cite[3.3.(1)]{MR753432}, the unit normal vector $\xi_{x}$ at $x$ is
\begin{equation}
    \xi_{x}
    =
    \frac{1}{\sqrt{\delta(1-c^2)}}
    \left\{
        (1-c)x
        -
        2
        \sum_{j=1}^{m}\eta_{jj}
        \langle P_{j}x, x\rangle
        P_{j}x
    \right\}
    \in T^{\bot}_{x}M_{c}
    ,\quad
    \langle \xi_{x}, \xi_{x}\rangle=\delta
    .
    \label{eq:nv}
\end{equation}

\subsection{The Focal Variety}
\label{sect:F-V}
Let 
\begin{align*}
    M_{+}
    &:=
    \left\{
        x\in S^{2l-1}_{s}
        \ \middle|\ 
        \left(
            \mathrm{grad}^{S^{2l-1}_{s}}(f)
        \right)(x)
        =
        0
        ,\ 
        f(x)=1
    \right\},\\
    M_{-}
    &:=
    \left\{
        x\in S^{2l-1}_{s}
        \ \middle|\ 
        \left(
            \mathrm{grad}^{S^{2l-1}_{s}}(f)
        \right)(x)
        =
        0
        ,\ 
        f(x)=-1
    \right\}.
\end{align*}
In this paper, we do not focus on $M_{-}$, but $M_{+}$.
(We refer to Remark \ref{rema:M-}.)
We refer to \cite[3.5. Proposition]{MR753432}.
We have
\begin{equation*}
    M_{+}
    =
    \left\{
        x\in S^{2l-1}_{s}
        \ \middle|\ 
        \forall j\in \{1, \ldots, m\}
        ,\ 
        \langle P_{j}x, x\rangle=0
    \right\}
    .
\end{equation*}
If $M_{+}\neq \emptyset$, then $M_{+}$ is a non-degenerate submanifold of codimension $m$ in $S^{2l-1}_{s}$.
In Riemannian case, we note that $M_{+}$ is connected.
(We refer to \cite[3.6, Connectedness of the level sets of $F$]{MR3408101}.)
From now on, we assume that $M_{+}\neq \emptyset$.
We have the $m$-dimensional normal space
$
    T^{\bot}_{x}M_{+}
    =
    \left\{
        Qx
        \ \middle|\ 
        Q\in \Sigma
    \right\}
$
at $x\in M_{+}$ of index $r$, since $\{P_{i}x\}_{i=1}^{m}$ is a pseudo-orthonormal basis of $T^{\bot}_{x}M_{+}$ by \S\ref{sect:C-S}\ref{metric}.
The normal bundle $T^{\bot}M_{+}$ is trivial, that is, $T^{\bot}M_{+}\cong M_{+}\times \mathbb{R}^{m}$.
Let 
\begin{equation*}
    \bot M_{+}(\varepsilon)
    :=
    \left\{
        (x, v)
        \ \middle|\ 
        x\in M_{+},\ 
        v\in T^{\bot}_{x}M_{+},\ 
        \langle v, v\rangle=\varepsilon
    \right\}
    ,\quad 
    \varepsilon\in \{-1, 1\}
    .
\end{equation*}
The set $\bot M_{+}(\delta)$ is non-empty since $r\neq 0$ if $\delta=-1$ and $r\neq m$ if $\delta=1$ by the definition of $M_{c}$, and has a structure of a fiber bundle.
We have $\bot M_{+}(1)\cong M_{+}\times S^{m-1}_{r}$ and $\bot M_{+}(-1)\cong M_{+}\times H^{m-1}_{r-1}$.

We assume that $c\in \mathrm{W}_{\mathrm{RN}}(f)\cap (-1,\infty)$.
We take the real number $t_{c}$ which satisfies $c=\cos{(4t)}$ and $t\in (0, \pi/4)$ if $\delta=1$, and $c=\cosh{(4t)}$ and $t\in (0, \infty)$ if $\delta=-1$.
We define $\varphi_{t_{c}}\colon\bot M_{+}(\delta)\to M_{c}$ by $\varphi_{t_{c}}(x, v):=\gamma_{x, v}(t_{c})$.
Referring to \cite[3.5.(2)]{MR753432}, $\varphi_{t_{c}}$ is a diffeomorphism. 
Thus, $\bot M_{+}(\delta)\cong M_{c}$ holds.
Moreover, $M_{+}$ is the focal variety associated to $\cot{t_{c}}$ if $\delta=1$, and $\coth{t_{c}}$ if $\delta=-1$.
(We refer to \cite[2.7]{MR753432}.)

Referring to \cite[3.5.(2), Proof]{MR753432}, for any $(x, v)\in \bot M_{+}(\delta)$, we have 
\begin{equation}
    f(\varphi_{t_{c}}(x, v))
    =
    \begin{cases}
        \cos{(4t_{c})}, & (\delta=1),\\
        \cosh{(4t_{c})}, & (\delta=-1).
    \end{cases}
    \label{lemm:65}
\end{equation}
Thus, we have 
\begin{equation}
    \mathrm{W}_{\mathrm{RN}}(f)
    \cap 
    (-1,\infty)
    =
    \begin{cases}
        (-1, 1), & (r=0),\\
        (-1, \infty)\setminus\{1\} & (r\in \{1, 2, \ldots, m-1\}),\\
        (1, \infty), & (r=m).
    \end{cases}
    \label{rema:63}
\end{equation}

Referring to \cite[3.5. (2), Proof]{MR753432}, we have
\begin{equation}
    \gamma_{x, v}'(t_{c})=-\xi_{\varphi_{t_{c}}(x, v)}\in T^{\bot}_{\varphi_{t_{c}}(x, v)}M_{c}
    ,\quad 
    (x, v)\in \bot M_{+}(\delta)
    ,
    \label{eq:tvnv}
\end{equation}
which can be proven by direct computation and \eqref{eq:nv}.

Let $(x, v)\in \bot M_{+}(\delta)$.
Then there exists uniquely $Q_{v}\in \Sigma$ such that $\langle Q_{v}, Q_{v}\rangle=\delta$ and $Q_{v}x=v$.
The shape operator associated to $v$ of $M_{+}$ is also denoted by $S_{v}$.
Referring to \cite[3.6. Proof]{MR753432}, we have 
\begin{equation}
    \ker{S_{v}}
    =
    \left\{
        QQ_{v}x
        \ \middle|\ 
        Q\in \Sigma
        ,
        \langle Q_{v}, Q\rangle=0
    \right\}
    .
    \label{ker}
\end{equation}
\subsection{Isometry and Homogeneity}
\label{sect:I-H}
A connected submanifold $L\subset S^{2l-1}_{s}$ is said to be \textit{homogeneous} if a Lie subgroup of the isometry group of $S^{2l-1}_{s}$ acts on $L$ transitively.
The isometry group of $S^{2l-1}_{s}$ is $O(s, 2l-s)$.
(We refer to \cite[9 Isometries, Some Isometry Groups]{MR719023}.)

We assume that there exists a Lie subgroup $K\subset O(s, 2l-s)$ such that $K$ acts on a connected non-degenerate submanifold $L\subset S^{2l-1}_{s}$.
Then the following facts hold:
\begin{enumerate}[label=(\arabic*)]
    \item \label{taniseibun}
    Let $K_{0}$ be the identity component of $K$.
    Then $K_{0}$ also acts on $L$ transitively.
    \item \label{tamotsu}For any $k\in K$ and $x\in L$, the differential $(dk)_{x}$ preserves tangent and normal components, that is, $(dk)_{x}(T_{x}L)=T_{k(x)}L$ and $(dk)_{x}(T^{\bot}_{x}L)=T^{\bot}_{k(x)}L$ hold.
\end{enumerate}

\section{A Construction of Clifford Systems}
In this section, we give a construction of Clifford systems of signature $(m, r)$ for any $(m, r)$.

\subsection{Representations of Clifford Algebras and Families of Orthogonal Matrices}

Referring to \cite[Chapter I, \S 4]{MR1031992}, we notice the existence of isomorphisms between Clifford algebras.
By using it, we can show that there exists a family of orthogonal matrices.
\begin{proposition}
    \label{prop:8}
    For any pair $(m, r)$ of a positive integer $m$ and $r\in \{0, 1, \ldots, m\}$, there exist a positive integer $l$ and a family of orthogonal matrices $\{A_{i}\}_{i=1}^{m}\ (A_{i}\in O(l))$ such that 
    \begin{equation}
        A_{i}A_{j}+A_{j}A_{i}=2\eta_{ij}E_{l}
        ,
        \label{eq:62}
    \end{equation}
    where $(\eta_{ij}):=J_{r, m-r}$.
\end{proposition}
To prove Proposition \ref{prop:8}, we will prepare some lemmas.
\begin{lemma}
    \label{lemm:12}
    Assume that there exist a positive integer $l$ and a family $\{A_{i}\}_{i=1}^{m}\ (A_{i}\in O(l))$ for $(m, r)$ satisfying \eqref{eq:62}.
    Then there exists $\{B_{j}\}_{j=1}^{m+2}\ (B_{j}\in O(2l))$ for $(m+2, r+1)$ satisfying \eqref{eq:62}.
\end{lemma}
\begin{proof}
    Let
    \begin{equation*}
        B_{j}
        :=
        \begin{cases}
            \begin{pmatrix}
                0&1\\
                1&0
            \end{pmatrix}
            \otimes A_{j},
            &(j\in \{1, 2, \ldots, r\}),\vspace{1truemm}\\
            \begin{pmatrix}
                0&1\\
                -1&0
            \end{pmatrix}
            \otimes E_{l},
            &(j=r+1),\vspace{1truemm}\\
            \begin{pmatrix}
                0&1\\
                1&0
            \end{pmatrix}
            \otimes A_{j-1},
            &(j\in \{r+2, r+3, \ldots, m+1\}),\vspace{1truemm}\\
            \begin{pmatrix}
                1&0\\
                0&-1
            \end{pmatrix}
            \otimes E_{l},
            &(j=m+2).
        \end{cases}
    \end{equation*}
    We can verify that $\{B_{j}\}$ satisfies \eqref{eq:62} by direct computation.
\end{proof}
\begin{lemma}
    \label{lemm:13}
    Assume that there exist a positive integer $l$ and a family $\{A_{i}\}_{i=1}^{m}\ (A_{i}\in O(l))$ for $(m, 0)$ satisfying \eqref{eq:62}.
    Then there exists $\{B_{j}\}_{j=1}^{m+2}\ (B_{j}\in O(4l))$ for $(m+2, m+2)$ satisfying \eqref{eq:62}.
\end{lemma}
\begin{proof}
    Let
    \begin{equation*}
        B_{j}
        :=
        \begin{cases}
            \begin{pmatrix}
                0&0&0&-1\\
                0&0&-1&0\\
                0&1&0&0\\
                1&0&0&0
            \end{pmatrix}
            \otimes A_{j},
            &(j\in \{1, 2, \ldots, m\}),\vspace{1truemm}\\
            \begin{pmatrix}
                0&-1&0&0\\
                1&0&0&0\\
                0&0&0&-1\\
                0&0&1&0
            \end{pmatrix}
            \otimes E_{l},
            &(j=m+1),\vspace{1truemm}\\
            \begin{pmatrix}
                0&0&-1&0\\
                0&0&0&1\\
                1&0&0&0\\
                0&-1&0&0
            \end{pmatrix}
            \otimes E_{l},
            &(j=m+2).
        \end{cases}
    \end{equation*}
    We can verify that $\{B_{j}\}$ satisfies \eqref{eq:62} by direct computation.
\end{proof}
\begin{lemma}
    \label{lemm:14}
    Assume that there exist a positive integer $l$ and a family $\{A_{i}\}_{i=1}^{m}\ (A_{i}\in O(l))$ for $(m, m)$ satisfying \eqref{eq:62}.
    Then there exists $\{B_{j}\}_{j=1}^{m+2}\ (B_{j}\in O(2l))$ for $(m+2, 0)$ satisfying \eqref{eq:62}.
\end{lemma}
\begin{proof}
    Let
    \begin{equation*}
        B_{j}
        :=
        \begin{cases}
            \begin{pmatrix}
                0&1\\
                -1&0
            \end{pmatrix}
            \otimes A_{j},
            &(j\in \{1, 2, \ldots, m\}),\vspace{1truemm}\\
            \begin{pmatrix}
                1&0\\
                0&-1
            \end{pmatrix}
            \otimes E_{l},
            &(j=m+1),\vspace{1truemm}\\
            \begin{pmatrix}
                0&1\\
                1&0
            \end{pmatrix}
            \otimes E_{l},
            &(j=m+2).
        \end{cases}
    \end{equation*}
    We can verify that $\{B_{j}\}$ satisfies \eqref{eq:62} by direct computation.
\end{proof}
\begin{lemma}
    \label{lemm:15}
    For $(m, r)=(1, 0), (1, 1), (2, 0), (2, 1)$ and $(2, 2)$, there exist a positive integer $l$ and a family of $\{A^{(m, r)}_{i}\}_{i=1}^{m}\ (A^{(m, r)}_{i}\in O(l))$ for $(m, r)$ satisfying \eqref{eq:62}.
\end{lemma}
\begin{proof}
    For $(m, r)=(1, 0)$, let $A^{(1, 0)}_{1}:=\begin{pmatrix}1\end{pmatrix}$.
    For $(m, r)=(1, 1)$, let
    \begin{equation*}
        A^{(1, 1)}_{1}:=
        \begin{pmatrix}
            0 & -1\\
            1 & 0
        \end{pmatrix}
        .
    \end{equation*}
    For $(m, r)=(2, 0)$, let
    \begin{equation*}
        A^{(2, 0)}_{1}
        :=
        \begin{pmatrix}
            1 & 0\\
            0 & -1
        \end{pmatrix}
        ,\quad
        A^{(2, 0)}_{2}
        :=
        \begin{pmatrix}
            0 & -1\\
            -1 & 0
        \end{pmatrix}
        .
    \end{equation*}
    For $(m, r)=(2, 1)$, let
    \begin{equation*}
        A^{(2, 1)}_{1}
        :=
        \begin{pmatrix}
            0 & 1\\
            -1 & 0
        \end{pmatrix}
        ,\quad
        A^{(2, 1)}_{2}
        :=
        \begin{pmatrix}
            1 & 0\\
            0 & -1
        \end{pmatrix}
        .
    \end{equation*}
    For $(m, r)=(2, 2)$, let
    \begin{equation*}
        A^{(2, 2)}_{1}
        :=
        \begin{pmatrix}
            0&-1&0&0\\
            1&0&0&0\\
            0&0&0&-1\\
            0&0&1&0
        \end{pmatrix}
        ,\quad
        A^{(2, 2)}_{2}
        :=
        \begin{pmatrix}
            0&0&-1&0\\
            0&0&0&1\\
            1&0&0&0\\
            0&-1&0&0
        \end{pmatrix}
        .
    \end{equation*}
    We can verify that $\{A^{(m, r)}_{i}\}$ satisfies \eqref{eq:62} by direct computation.
\end{proof}
\begin{proof}
    (The Proof of Proposition \ref{prop:8})
    We prove Proposition \ref{prop:8} by induction on $m$.
    For $m=1$ and $m=2$, we have such families of orthogonal matrices satisfying \eqref{eq:62} by Lemma \ref{lemm:15}.
    Next, we assume that there are families of orthogonal matrices satisfying \eqref{eq:62} for $m=2b-1$ and $m=2b$.
    By Lemma \ref{lemm:13} and using the families of orthogonal matrices for $(m, r)=(2b-1, 0), (2b, 0)$, we have families of orthogonal matrices for $(m, r)=(2b+1, 2b+1), (2b+2, 2b+2)$.
    By Lemma \ref{lemm:14} and using the families of orthogonal matrices for $(m, r)=(2b-1, 2b-1), (2b, 2b)$, we have families of orthogonal matrices for $(m, r)=(2b+1, 0), (2b+2, 0)$.
    By Lemma \ref{lemm:12} and using the families of orthogonal matrices for $(m, r)=(2b-1, 0), (2b-1, 1), \ldots, (2b-1, 2b-1), (2b, 0), (2b, 1), \ldots, (2b, 2b)$, we have families of orthogonal matrices for $(m, r)=(2b+1, 1), (2b+1, 2), \ldots, (2b+1, 2b), (2b+2, 1), (2b+2, 2), \ldots, (2b+2, 2b+1)$.
    Thus there are families of orthogonal matrices satisfying \eqref{eq:62} for $m=2b+1$ and $m=2b+2$.
    This completes the proof by induction on $m$.
\end{proof}
\begin{remark}
    \label{rema:reduce}
    We can reduce the order $l$ in Proposition \ref{prop:8}.
    Referring to \cite{MR2251632}, we can construct a family of orthogonal matrices of the following order $l(m)$ for $(m, m)\ (m=1, 2, \ldots, 8, 9, \ldots)$ satisfying \eqref{eq:62}:
    \begin{equation*}
        \begin{array}{|c||c|c|c|c|c|c|c|c|c|c|}\hline
            m&0&1&2&3&4&5&6&7&\cdots&i\\\hline
            l(m)&1&2&4&4&8&8&8&8&\cdots&16\cdot l(i-8)\\\hline
        \end{array}
    \end{equation*}
    These are used to construct Clifford systems in Riemannian case.
    (We refer to \cite[3.5]{MR624227} or \cite[3.9, p.163]{MR3408101}.)
    We can construct a family of orthogonal matrices for $(m, 0)$ satisfying \eqref{eq:62} by Lemma \ref{lemm:14} and for $(m, r)\ (r\in \{1, 2, \ldots, m-1\})$ satisfying \eqref{eq:62} by Lemma \ref{lemm:12}.
    In this way, the orders of these families of orthogonal matrices are smaller than the orders of ones in Proposition \ref{prop:8}.
\end{remark}

\subsection{Clifford Systems}
We construct Clifford systems of signature $(m, r)$ for any $(m, r)$ explicitly.
\begin{lemma}
    \label{lemm:16}
    For $\{A_{i}\}_{i=1}^{m}\ (A_{i}\in O(l))$ satisfying \eqref{eq:62}, $A_{i}$ is skew-symmetric if $i\in \{1, 2, \ldots r\}$, or $A_{i}$ is symmetric if $i\in \{r+1, r+2, \ldots, m\}$.
\end{lemma}
\begin{proof}
    If $i\in \{1, 2, \ldots, r\}$, we have $(A_{i})^{2}=-E_{l}$.
    We multiply the both side by ${}^{t}\!A_{i}=(A_{i})^{-1}$.
    Then we have ${}^{t}\!A_{i}=-A_{i}$.
    Similarly, if $i\in \{r+1, r+2, \ldots, m\}$, we have ${}^{t}\!A_{i}=A_{i}$.
\end{proof}
Proposition \ref{prop:9} follows from Lemma \ref{lemm:16}.
\begin{proposition}
    \label{prop:9}
    Let $m$ be an integer lager than $1$ and $r\in \{0, 1, 2, \ldots, m\}$.
    For $\{A_{i}\}_{i=1}^{m}\ (A_{i}\in O(l))$ satisfying \eqref{eq:62} and for a positive integer $d$, let $\{P_{i}\}_{i=1}^{m}$ be the matrices
    \begin{equation*}
        P_{i}
        :=
        \begin{cases}
            \begin{bmatrix}
                &&A_{i}\\
                &\iddots&\\
                A_{i}&&
            \end{bmatrix}
            ,&(i\in \{1, 2, \ldots, r\}),\vspace{1truemm}\\
            \begin{bmatrix}
                A_{i}&&\\
                &\ddots&\\
                &&A_{i}
            \end{bmatrix}
            ,&(i\in \{r+1, r+2, \ldots, m\}),\\
        \end{cases}
    \end{equation*}
    where the number of $A_{i}$ in $P_{i}$ is $2d$.
    Then $\{P_{i}\}_{i=1}^{m}$ is a Clifford system of signature $(m, r)$ on $\mathbb{R}^{2dl}_{dl}$.
\end{proposition}
By Propositions \ref{prop:8} and \ref{prop:9}, we obtain concrete Clifford systems of signature $(m, r)$ for any $(m, r)$.
\begin{remark}
    By Proposition \ref{prop:9}, we make Clifford systems of signature $(m, r)$ on $\mathbb{R}^{2dl}_{dl}$, which has the neutral metric.
    However, Proposition \ref{prop:9} is meaningful by \S\ref{sect:C-S}\ref{index-equiv}.
\end{remark}
\begin{theorem}
    \label{theo:C-exist}
    For any pair $(m, r)$ of an integer $m$ larger than $1$ and $r\in \{0, 1, 2, \ldots, m\}$, there exist a positive integer $l$ and a Clifford system $\{P_{i}\}_{i=1}^{m}\ (P_{i}\in \mathrm{Sym}(\mathbb{R}^{2l}_{l}))$ of signature $(m, r)$.
\end{theorem}
\begin{remark}
    In Riemannian case, it is easy that we find the order $l$ of a Clifford system which has $m$ matrices since representations of Clifford algebras have periodicity of its degree.
    (We refer to \cite[3.9, p.163]{MR3408101} and Remark \ref{rema:reduce}.)
    On the other hand, in pseudo-Riemannian case, we can find the order $l$ in Proposition \ref{prop:9} by Lemmas \ref{lemm:12}, \ref{lemm:13}, \ref{lemm:14} and \ref{lemm:15}.
\end{remark}

\section{Inhomogeneity of Isoparametric Hypersurfaces of OT-FKM-type}
In this section, we investigate inhomogeneity of isoparametric hypersurfaces of OT-FKM-type.

\subsection{Orthogonal Decompositions}
\label{sect:4-1}
Let $\{P_{i}\}_{i=1}^{m}$ be a Clifford system of signature $(m, r)$ satisfying $P_{i}P_{j}+P_{j}P_{i}=2\eta_{ij}I$ for $i, j\in \{1, \ldots, m\}$.
Hereafter, we assume that $m\equiv 0\pmod{4}$ and $r\equiv 0\pmod{2}$.
Let $p\in \{4, 8, \ldots m\}$ and $q\in \{1, 3, \ldots, m-3\}$ where $q+p-1\leqq m$.
Let $Q_{q, p}:=P_{q}P_{q+1}\cdots P_{q+p-1}$.
\begin{lemma}
    \label{lemm:17-1} 
    We have $Q_{q, p}=P_{q+p-1}P_{q+p-2}\cdots P_{q}$.
\end{lemma}
\begin{proof}
    By \S \ref{sect:C-S}\ref{change-PQ}, we have 
    $Q_{q, p}=(-1)^{{p(p-1)}/{2}}P_{q+p-1}P_{q+p-2}\cdots P_{q}$.
    Since $p\equiv 0\pmod{4}$, we have $(-1)^{{p(p-1)}/{2}}=1$.
\end{proof}
\begin{lemma}
    \label{lemm:17-2} 
    We have $Q_{q, p}\in \mathrm{Sym}(\mathbb{R}^{2l}_{s})$.
\end{lemma}
\begin{proof}
    By Lemma \ref{lemm:17-1}, we have
    $
        \langle Q_{q, p}u, v\rangle
        =
        \langle u, P_{q+p-1}P_{q+p-2}\cdots P_{q}v\rangle
        =
        \langle u, Q_{q, p}v\rangle
    $
    for $u, v\in \mathbb{R}^{2l}_{s}$.
\end{proof}
\begin{lemma}
    \label{lemm:17-3} 
    We have $(Q_{q, p})^{2}=I$, that is, $Q_{q, p}$ is involutive.
\end{lemma}
\begin{proof}
    By Lemma \ref{lemm:17-1}, we have
    $
        (Q_{q, p})^{2}
        =
        Q_{q, p}
        P_{q+p-1}P_{q+p-2}\cdots P_{q}
        =
        \eta_{qq}
        \eta_{q+1, q+1}
        \cdots
        \eta_{q+p-1, q+p-1}
        I
        .
    $
    Since $r$ is even, we have $\eta_{qq}\eta_{q+1, q+1}\cdots\eta_{q+p-1, q+p-1}=1$ .
\end{proof}
\begin{lemma}
    \label{lemm:17-4} 
    For any $a\in \{q, q+1, \ldots, q+p-1\}$, $P_{a}Q_{q, p}=-Q_{q, p}P_{a}$ holds.
\end{lemma}
\begin{proof}
    By \S \ref{sect:C-S}\ref{change-PQ}, we have 
    $
        P_{a}Q_{q, p}
        =
        (-1)^{p-1}
        Q_{q, p}P_{a}
        .
    $
    Since $p$ is even, we have $(-1)^{p-1}=-1$.
\end{proof}
Let $Q\in \mathrm{End}(\mathbb{R}^{2l}_{s})$ be a linear mapping whose eigenvalues are $-1$ and $1$.
Let $E_{+}(Q)$ and $E_{-}(Q)$ be the eigenspaces of $Q$ associated with $1$ and $-1$, respectively.
For simplicity, we write $E_{+}(Q)$ and $E_{-}(Q)$ as $E_{1}(Q)$ and $E_{-1}(Q)$, respectively.
\begin{proposition}
    \label{prop:18}
    All of the eigenvalues of $Q_{q, p}$ are $-1$ and $1$.
    We have the orthogonal decomposition $\mathbb{R}^{2l}_{s}=E_{+}(Q_{q, p})\oplus E_{-}(Q_{q, p})$.
    Every $z\in \mathbb{R}^{2l}_{s}$ can be uniquely expressed as $z=z_{+}+z_{-}\ (z_{+}\in E_{+}(P),\ z_{-}\in E_{-}(P))$.
    For $\varepsilon\in \{-1, 1\}$, the eigenspace $E_{\varepsilon}(Q_{q, p})$ is an $l$-dimensional non-degenerate subspace of $\mathbb{R}^{2l}_{s}$.
\end{proposition}
\begin{proof}
    By Lemma \ref{lemm:17-3}, we have the decomposition $\mathbb{R}^{2l}_{s}=E_{+}(Q_{q, p})\oplus E_{-}(Q_{q, p})$.
    By Lemma \ref{lemm:17-4}, we can verify $\dim{E_{+}(Q_{q, p})}\geqq 1$ and $\dim{E_{-}(Q_{q, p})}\geqq 1$.
    Thus, all of the eigenvalues of $Q_{q, p}$ are $-1$ and $1$.

    By Lemma \ref{lemm:17-2}, for $u\in E_{+}(Q_{q, p})$ and $v\in E_{-}(Q_{q, p})$, we have
    \begin{equation*}
        \langle u, v\rangle
        =
        \langle Q_{q, p}u, v\rangle
        =
        \langle u, Q_{q, p}v\rangle
        =
        -\langle u, v\rangle
        .
    \end{equation*}
    Thus, $\langle u, v\rangle=0$ holds.
    Hence, the decomposition is the orthogonal one.

    Let $\{u_{i}\}_{i=1}^{d_{+}}$ and $\{v_{i}\}_{i=1}^{d_{-}}$ be bases of $E_{+}(Q_{q, p})$ and $E_{-}(Q_{q, p})$, respectively.
    By Lemma \ref{lemm:17-4}, $P_{q+p-1}u_{i}\in E_{-}(Q_{q, p})$ for all $i\in \{1, \ldots, d_{+}\}$.
    Since $P_{q+p-1}$ is regular, the sequence of vectors $\{P_{q+p-1}u_{i}\}_{i=1}^{d_{+}}$ of $E_{-}(Q_{q, p})$ is linearly independent and we have $d_{+}\leqq d_{-}$.
    Similarly, we have $d_{-}\leqq d_{+}$.
    Thus, we have $d_{+}=d_{-}$.
    Since $d_{+}+d_{-}=2l$, we have $d_{+}=d_{-}=l$.
\end{proof}
\begin{remark}
    \label{rema:51}
    $E_{+}(Q_{q, p})\cap S^{2l-1}_{s}\neq \emptyset$ or $E_{-}(Q_{q, p})\cap S^{2l-1}_{s}\neq\emptyset$ holds.
    In fact, let $s_{1}$ and $s_{2}$ be the signatures of $E_{+}(Q_{q, p})$ and $E_{-}(Q_{q, p})$, respectively.
    Since $s\leqq 2l-1$, $s_{1}\leqq l-1$ or $s_{2}\leqq l-1$ holds.
\end{remark}
\begin{remark}
    \label{rema:lm}
    $l\geqq m$ holds.
    In fact, by Remark \ref{rema:51}, we can take $x\in E_{\varepsilon}(P)\cap S^{2l-1}_{s}$ ($\varepsilon=1$ or $\varepsilon=-1$).
    By Lemma \ref{lemm:17-4}, $Qx\in E_{-\varepsilon}(P)$ holds for any $Q\in \Sigma$.
    Thus, the mapping $\Sigma\ni Q\mapsto Qx\in E_{-\varepsilon}(P)$ is an injective isometric linear map.
    Hence, we have $m=\dim{\Sigma}\leqq \dim{E_{-\varepsilon}(P)}=l$.

    If $m=l$, the number of the principal curvatures of isoparametric hypersurfaces of OT-FKM-type is two.
    (We refer to \cite[3.8.(3), (4), (5)]{MR753432}.)
\end{remark}

\subsection{Connectedness of the Focal Variety}
\label{sect:4-2}
We use the settings in \S\ref{sect:4-1}.
We investigate the connectedness of $M_{+}$ under the assumption.
We set $P:=Q_{1, m}=P_{1}P_{2}\cdots P_{m}$.
Let $s_{1}$ and $s_{2}$ be the signatures of $E_{+}(P)$ and $E_{-}(P)$, respectively.
Then $s_{1}+s_{2}=s$ holds by Proposition \ref{prop:18}.
\begin{lemma}
    \label{lemm:52}
    The signatures $s_{1}$ and $s_{2}$ satisfy the following:
    \begin{enumerate}[label=(\arabic*)]
        \item \label{lemm:52-1}
        If $r\in \{0, 2, 4, \ldots, m-2\}$, then $s$ is even and $s_{1}=s_{2}=s/2$ holds.
        \item \label{lemm:52-2}
        If $r=0$, then $s/2\leqq l-m$.
        \item \label{lemm:52-3}
        If $r\in \{2, 4, \ldots, m-2\}$, then $m-r\leqq l/2=s/2$.
        \item \label{lemm:52-4}
        If $r=m$, then one of $(s_{1}, s_{2})=(0, l)$; $m\leqq s_{1}\leqq l-m$ and $m\leqq s_{2}\leqq l-m$; or $(s_{1}, s_{2})=(l, 0)$ holds.
    \end{enumerate}
\end{lemma}
\begin{proof}
    \begin{enumerate}[label=(\arabic*)]
        \item 
        Let $\{x_{i}\}_{i=1}^{l}$ be a pseudo-orthonormal basis of $E_{+}(P)$.
        By Lemma \ref{lemm:17-4}, $\{P_{m}x_{i}\}_{i=1}^{l}$ is a pseudo-orthonormal basis of $E_{-}(P)$.
        In fact, we have 
        $
            \langle P_{m}x_{i}, P_{m}x_{j}\rangle
            =
            \langle x_{i}, (P_{m})^{2}x_{j}\rangle
            =
            \langle x_{i}, x_{j}\rangle
        $
        since $r\neq m$.
        Then we have $s_{1}=s_{2}$.
        Since $s_{1}+s_{2}=s$, we have $s_{1}=s_{2}=s/2$.
        \item By Remark \ref{rema:51}, we can take $x\in E_{\varepsilon}(P)\cap S^{2l-1}_{s}$ ($\varepsilon=1$ or $\varepsilon=-1$).
        By Lemma \ref{lemm:17-4} and \S\ref{sect:C-S}\ref{metric}, $\{P_{j}x\}_{j=1}^{m}$ is a pseudo-orthonormal system of $E_{-\varepsilon}(P)$.
        We note that $\langle P_{j}x, P_{j}x\rangle=1\ (j\in \{1, \ldots, m\})$ holds.
        By \ref{lemm:52-1}, we have $m\leqq l-s/2$.
        \item The proof is similar to \ref{lemm:52-2}.
        By \S\ref{sect:C-S}\ref{index-equiv}, we have $m-r\leqq l-s/2=l/2$.
        \item 
        By \S\ref{sect:C-S}\ref{index-equiv}, $s_{1}+s_{2}=s=l$ holds.
        Thus, $(s_{1}, s_{2})=(0, l), (l, 0)$ could occur.
        We assume that $s_{1}>0$ and $s_{2}>0$.
        Since $0<s_{1}=l-s_{2}<l$ and $0<s_{2}=l-s_{1}<l$ hold, we can take $x\in E_{+}(P)\cap S^{2l-1}_{s}$ and $y\in E_{-}(P)\cap S^{2l-1}_{s}$.
        By Lemma \ref{lemm:17-4} and \S\ref{sect:C-S}\ref{metric}, $\{P_{j}x\}_{j=1}^{m}$ and $\{P_{j}y\}_{j=1}^{m}$ are pseudo-orthonormal systems of $E_{-}(P)$ and $E_{+}(P)$, respectively.
        We note that $\langle P_{j}x, P_{j}x\rangle=-1\ (j\in \{1, \ldots, m\})$ and $\langle P_{j}y, P_{j}y\rangle=-1\ (j\in \{1, \ldots, m\})$ hold.
        Then we have $m\leqq s_{1}=l-s_{2}$ and $m\leqq s_{2}=l-s_{1}$.
    \end{enumerate}
\end{proof}
We divide into the following cases:
\begin{enumerate}[label=(\alph*)]
    \item \label{case:a}
    $r=0$ and $s/2=l-m$.
    \item \label{case:b}
    $r\in \{2, 4, \ldots, m-2\}$ and $0<m-r=l/2$.
    \item \label{case:c}
    \begin{enumerate}[label=(c-\arabic*)]
        \item \label{case:c-1}
        $r=m$ and $(s_{1}, s_{2})=(0, l)$.
        \item \label{case:c-2}
        $r=m$ and $(s_{1}, s_{2})=(l, 0)$.
    \end{enumerate}
    \item \label{case:d}
    \begin{enumerate}[label=(d-\arabic*)]
        \item \label{case:d-1}
        $r=0$ and $s/2<l-m$.
        \item \label{case:d-2}
        $r\in \{2, 4, \ldots, m-2\}$ and $m-r<l/2$.
        \item \label{case:d-3}
        $r=m$, $s_{1}+s_{2}=l$, $m\leqq s_{1}\leqq l-m$ and $m\leqq s_{2}\leqq l-m$.
    \end{enumerate}
\end{enumerate}
\ref{case:a}, \ref{case:b}, \ref{case:c-1}, \ref{case:c-2}, \ref{case:d-1}, \ref{case:d-2} and \ref{case:d-3} are all cases under the assumption by Lemma \ref{lemm:52}.
\begin{lemma}
    \label{lemm:54}
    Every $z\in \mathbb{R}^{2l}_{s}$ can be uniquely expressed as $z=z_{+}+z_{-}\ (z_{+}\in E_{+}(P),\ z_{-}\in E_{-}(P))$.
    Then $z\in M_{+}$ if and only if $\langle z_{+}, z_{+}\rangle+\langle z_{-}, z_{-}\rangle=1$ and $\langle P_{j}z_{+}, z_{-}\rangle=0$ for all $j\in \{1, \ldots, m\}$. 
\end{lemma}
\begin{proof}
    By Proposition \ref{prop:18}, $z\in S^{2l-1}_{s}$ if and only if $\langle z_{+}, z_{+}\rangle+\langle z_{-}, z_{-}\rangle=1$.
    We fix $j\in \{1,\ldots, m\}$.
    By Lemma \ref{lemm:17-4}, we note that $P_{j}z_{+}\in E_{-}(P)$ and $P_{j}z_{-}\in E_{+}(P)$ hold.
    We can compute 
    \begin{align*}
        \langle P_{j}z, z\rangle
        &=
        \langle P_{j}z_{+}, z_{+}\rangle
        +
        2\langle P_{j}z_{+}, z_{-}\rangle
        +
        \langle P_{j}z_{-}, z_{-}\rangle
        =
        2\langle P_{j}z_{+}, z_{-}\rangle
        .
    \end{align*}
\end{proof}
We define 
\begin{align*}
    M_{+, 1}
    &:=
    \left\{
        z\in M_{+}
        \ \middle|\ 
        z=z_{+}+z_{-},\ 
        z_{+}\in E_{+}(P),\ 
        z_{-}\in E_{-}(P),\ 
        \langle z_{+}, z_{+}\rangle\geqq 1
    \right\}
    ,\\
    M_{+, 2}
    &:=
    \left\{
        z\in M_{+}
        \ \middle|\ 
        z=z_{+}+z_{-},\  
        z_{+}\in E_{+}(P),\ 
        z_{-}\in E_{-}(P),\ 
        \langle z_{+}, z_{+}\rangle\leqq 0
    \right\}
    ,\\
    M_{+, 3}
    &:=
    \left\{
        z\in M_{+}
        \ \middle|\ 
        z=z_{+}+z_{-},\  
        z_{+}\in E_{+}(P),\ 
        z_{-}\in E_{-}(P),\ 
        0<\langle z_{+}, z_{+}\rangle<1
    \right\}
    .
\end{align*}
\begin{remark}
    \label{rema:53}
    We have $M_{+}=M_{+, 1}\sqcup M_{+, 2}\sqcup M_{+, 3}$ by the definition of $M_{+, 1}$, $M_{+, 2}$ and $M_{+, 3}$.
\end{remark}
\begin{lemma}
    \label{lemm:56}
    $E_{+}(P)\cap S^{2l-1}_{s}\subset M_{+, 1}$ and $E_{-}(P)\cap S^{2l-1}_{s}\subset M_{+, 2}$ hold.
    In particular, $M_{+}$ is non-empty.
\end{lemma}
\begin{proof}
    By Lemma \ref{lemm:54}, if $z\in E_{+}(P)\cap S^{2l-1}_{s}$, then $z\in M_{+}$.
    Since $\langle z, z\rangle=1$, $z\in M_{+, 1}$ holds.
    The proof of $E_{-}(P)\cap S^{2l-1}_{s}\subset M_{+, 2}$ is similar to the former.
    By Remark \ref{rema:51}, $M_{+}\neq \emptyset$ holds.
\end{proof}
By Lemma \ref{lemm:56} and \S\ref{sect:F-V}, $M_{+}$ is a non-degenerate submanifold of codimension $m$ in $S^{2l-1}_{s}$.
\begin{lemma}
    \label{lemm:57}
    For $\varepsilon\in \{-1, 1\}$, if $E_{\varepsilon}(P)\cap S^{2l-1}_{s}$ is non-empty, then $E_{\varepsilon}(P)\cap S^{2l-1}_{s}$ is path-connected.
\end{lemma}
\begin{proof}
    We consider the case of $\varepsilon=1$.
    Since $E_{+}(P)\cap S^{2l-1}_{s}$ is non-empty, it is necessary that $s_{1}\leqq l-1$.
    Here, $E_{+}(P)\cap S^{2l-1}_{s}\cong S^{l-1}_{s_{1}}$ holds.
    Since $S^{l-1}_{s_{1}}$ is disconnected if and only if $s_{1}=l-1$, it is suffices to show that $s_{1}<l-1$.
    In \ref{case:a} and \ref{case:d-1}, by Lemma \ref{lemm:52}\ref{lemm:52-1}, \ref{lemm:52-2}, we have $s_{1}=s/2\leqq l-m\leqq l-4<l-1$.
    In \ref{case:b} and \ref{case:d-2}, we suppose that $s_{1}\geqq l-1$ holds.
    By Lemma \ref{lemm:52}\ref{lemm:52-1} and \S\ref{sect:C-S}\ref{index-equiv}, we have $l-1\leqq s_{1}=s/2=l/2$.
    By Remark \ref{rema:lm}, we have $2\geqq l\geqq m\geqq 4$, which is a contradiction.
    By Lemma \ref{lemm:52}\ref{lemm:52-4}, the following claims are true: In \ref{case:c-1}, $s_{1}<l-1$ holds; in \ref{case:d-3}, we have
    $
        s_{1}\leqq l-m\leqq l-4<l-1;
    $
    the case \ref{case:c-2} do not occur under the assumption.

    The proof of the case of $\varepsilon=-1$ is similar to the one for $\varepsilon=1$.
\end{proof}
\begin{lemma}
    \label{lemm:60}
    We assume that $M_{+, 3}\neq \emptyset$.
    We fix $z\in M_{+, 3}$.
    The element $z$ can be uniquely expressed as $z=z_{+}+z_{-}\ (z_{+}\in E_{+}(P),\ z_{-}\in E_{-}(P))$.
    Since $z\in M_{+, 3}$, we have $\langle z_{+}, z_{+}\rangle, \langle z_{-}, z_{-}\rangle\in (0, 1)$ by Lemma \ref{lemm:54}.
    Let
    \begin{equation*}
        x
        :=
        \frac{z_{+}}{\sqrt{\langle z_{+}, z_{+}\rangle}}
        \in E_{+}(P)\cap S^{2l-1}_{s}
        ,\quad
        y
        :=
        \frac{z_{-}}{\sqrt{\langle z_{-}, z_{-}\rangle}}
        \in E_{-}(P)\cap S^{2l-1}_{s}
        .
    \end{equation*}
    Then there exists a path in $M_{+}$ from $x$ to $y$ through $z$.
    Moreover, $\langle P_{j}x, y\rangle=0$ holds for all $j\in \{1, \ldots, m\}$.
\end{lemma}
\begin{proof}
    We take $t_{z}\in (0, \pi/2)$ uniquely such that 
    \begin{equation*}
        \cos{t_{z}}=\sqrt{\langle z_{+}, z_{+}\rangle}
        ,\ 
        \sin{t_{z}}=\sqrt{\langle z_{-}, z_{-}\rangle}
        .
    \end{equation*}
    The curve $\alpha\colon[0, \pi/2]\to M_{+}$ is given by 
    \begin{equation*}
        \alpha(t)
        :=
        (\cos{t})x+(\sin{t})y
        .
    \end{equation*}
    Then $\alpha$ is well-defined.
    In fact, we can compute $\langle (\cos{t})x, (\cos{t})x\rangle+\langle (\sin{t})y, (\sin{t})y\rangle=1$. 
    Since $z\in M_{+}$, we have $\langle P_{j}z_{+}, z_{-}\rangle=0$ by Lemma \ref{lemm:54}.
    Thus, we have 
    \begin{equation}
        \langle P_{j}(\cos{t})x, (\sin{t})y\rangle
        =
        \frac{\cos{t}\cdot \sin{t}}{\sqrt{\langle z_{+}, z_{+}\rangle}\sqrt{\langle z_{-}, z_{-}\rangle}}
        \langle P_{j}z_{+}, z_{-}\rangle
        =
        0
        ,\quad
        j\in \{1, 2, \ldots, m\}
        .
        \label{eq:60}
    \end{equation}
    By Lemma \ref{lemm:54}, we have $\alpha(t)\in M_{+}$ for any $t\in [0, \pi/2]$.
    Since $\alpha(0)=x$, $\alpha(t_{z})=z$ and $\alpha(\pi/2)=y$, the curve $\alpha$ is a path from $x$ to $y$ through $z$.
    
    By the equation \eqref{eq:60} with $t=\pi/4$, we have $\langle P_{j}x, y\rangle=0$ for all $j\in \{1, \ldots, m\}$.
\end{proof}
\begin{lemma}
    \label{lemm:58}
    For $\varepsilon\in \{-1, 1\}$, if $M_{+, (3-\varepsilon)/2}$ is non-empty, then $E_{\varepsilon}(P)\cap S^{2l-1}_{s}$ is non-empty and $M_{+, (3-\varepsilon)/2}$ is path-connected.
\end{lemma}
\begin{proof}
    We consider the case of $\varepsilon=1$.
    We show that there exists a path in $M_{+, 1}$ from each point of $M_{+, 1}$ to a point of $E_{+}(P)\cap S^{2l-1}_{s}$.
    We take an arbitrary point $z\in M_{+, 1}$.
    The element $z$ can be uniquely expressed as $z=z_{+}+z_{-}\ (z_{+}\in E_{+}(P),\ z_{-}\in E_{-}(P))$.
    
    We consider the case where $\langle z_{+}, z_{+}\rangle>1$.
    By Lemma \ref{lemm:54}, we have $\langle z_{-}, z_{-}\rangle<0$.
    Let 
    \begin{equation*}
        x
        :=
        \frac{z_{+}}{\sqrt{\langle z_{+}, z_{+}\rangle}}
        ,\quad
        y
        :=
        \frac{z_{-}}{\sqrt{-\langle z_{-}, z_{-}\rangle}}
        .
    \end{equation*}
    We take $t_{z}\in (0, \infty)$ uniquely such that 
    \begin{equation*}
        \cosh{t_{z}}=\sqrt{\langle z_{+}, z_{+}\rangle}
        ,\quad
        \sinh{t_{z}}=\sqrt{-\langle z_{-}, z_{-}\rangle}
        .
    \end{equation*}
    The curve $\alpha_{1}\colon[0, t_{z}]\to M_{+, 1}$ is given by 
    \begin{equation*}
        \alpha_{1}(t)
        :=
        (\cosh{t})x+(\sinh{t})y
        .
    \end{equation*}
    We can verify that $\alpha_{1}(t)\in M_{+, 1}\ (\forall t\in [0, t_{z}])$, $\alpha_{1}(t_{z})=z$ and $\alpha_{1}(0)=x$.
    (We refer to the proof of Lemma \ref{lemm:60}.)
    Thus, $\alpha_{1}$ is a path in $M_{+, 1}$ from $x$ to $z$.

    We consider case where $\langle z_{+}, z_{+}\rangle=1$.
    The curve $\alpha_{2}\colon[0, 1]\to M_{+, 1}$ is given by 
    \begin{equation*}
        \alpha_{2}(t)=z_{+}+tz_{-}
        .
    \end{equation*}
    We can verify that $\alpha_{2}(t)\in M_{+, 1}\ (\forall t\in [0, 1])$, $\alpha_{2}(1)=z$ and $\alpha_{2}(0)=z_{+}$.
    (We refer to the proof of Lemma \ref{lemm:60}.)
    Thus, $\alpha_{2}$ is a path in $M_{+, 1}$ from $z_{+}$ to $z$.
    
    Since $x$ in the first case and $z_{+}$ in the second case are elements of $E_{+}(P)\cap S^{2l-1}_{s}$, the set $E_{+}(P)\cap S^{2l-1}_{s}$ is non-empty.
    Since $z$ is arbitrary, we can verify that $M_{+, 1}$ is path-connected by Lemma \ref{lemm:57}.

    The proof of the case of $\varepsilon=-1$ is similar to the one of $\varepsilon=1$.
\end{proof}
\begin{proposition}
    \label{prop:51}
    In \ref{case:a} and \ref{case:b}, $M_{+}$ is disconnected.
    All of the connected components of $M_{+}$ are $M_{+, 1}$ and $M_{+, 2}$.
\end{proposition}
\begin{proof}
    We suppose that $M_{+, 3}\neq \emptyset$.
    By Lemma \ref{lemm:60}, there exist $x\in E_{+}(P)\cap S^{2l-1}_{s}$ and $y\in E_{-}(P)\cap S^{2l-1}_{s}$ such that $\langle P_{j}x, y\rangle=0$ holds for all $j\in \{1, \ldots, m\}$.
    By Lemma \ref{lemm:17-4} and \S\ref{sect:C-S}\ref{metric}, $\{P_{j}x\}_{j=r+1}^{m}\cup \{y\}$ is a pseudo-orthonormal system of $E_{-}(P)$.
    We note that $\langle P_{j}x, P_{j}x\rangle=1\ (j\in \{r+1, \ldots, m\})$ and $\langle y, y\rangle=1$.
    In \ref{case:a}, we have $m+1\leqq l-s/2=m$, which is a contradiction.
    In \ref{case:b}, we have $m-r+1\leqq l/2=m-r$, which is a contradiction.
    Therefore, $M_{+, 3}=\emptyset$ holds.
    By Remark \ref{rema:51}, there exists $x\in E_{\varepsilon}(P)\cap S^{2l-1}_{s}$ ($\varepsilon=1$ or $\varepsilon=-1$).
    We note that $r\neq m$ in \ref{case:a} and \ref{case:b}.
    By Lemma \ref{lemm:17-4} and \S\ref{sect:C-S}\ref{metric}, we have $P_{m}x\in E_{-\varepsilon}(P)\cap S^{2l-1}_{s}$.
    By Remark \ref{rema:53} and Lemma \ref{lemm:56}, we have
    \begin{equation}
        M_{+}
        =
        M_{+, 1}\sqcup M_{+, 2}
        ,\quad
        M_{+, 1}\neq \emptyset
        ,\quad
        M_{+, 2}\neq \emptyset
        \label{eq:81}
    \end{equation}

    We suppose that $M_{+}$ is path-connected.
    By \eqref{eq:81}, we take arbitrary $z_{1}\in M_{+, 1}$ and $z_{2}\in M_{+, 2}$.
    Since $M_{+}$ is path-connected, there exists a path $\alpha\colon[0, 1]\to M_{+}\subset \mathbb{R}^{2l}_{s}$ from $z_{1}$ to $z_{2}$.
    The mapping $\mathcal{P}_{1}\colon \mathbb{R}^{2l}_{s}\to E_{+}(P)$ is given by $\mathcal{P}_{1}:=(I+P)/2$, which is a projection.
    Since $\alpha$ is continuous, the mapping $\mathcal{P}_{1}\circ \alpha\colon [0, 1]\to E_{+}(P)$ is continuous.
    The function $g\colon [0, 1]\to \mathbb{R}$ is given by $g(t):=\langle\mathcal{P}_{1}\circ \alpha(t), \mathcal{P}_{1}\circ \alpha(t)\rangle$, which is continuous on $[0, 1]$.
    On the other hand, $z_{1}$ and $z_{2}$ can be uniquely expressed as $z_{1}=z_{1, +}+z_{1, -}$ and $z_{2}=z_{2, +}+z_{2, -}\ (z_{1, +}, z_{2, +}\in E_{+}(P),\ z_{1, -}, z_{2, -}\in E_{-}(P))$.
    By Lemma \ref{lemm:54}, we have 
    \begin{equation*}
        g(0)=\langle z_{1, +}, z_{1, +}\rangle\geqq 1
        ,\quad
        g(1)=\langle z_{2, +}, z_{2, +}\rangle\leqq 0
        .
    \end{equation*}
    By the intermediate value theorem, there exists $t_{\mathrm{mid}}\in (0, 1)$ such that $g(t_{\mathrm{mid}})=1/2$.
    Then $\alpha(t_{\mathrm{mid}})\in M_{+, 3}$.
    This is a contradiction by \eqref{eq:81}.
    Therefore, $M_{+}$ is not path-connected.
    Since $M_{+}$ is a manifold, $M_{+}$ is disconnected.
    By Lemma \ref{lemm:58} and \eqref{eq:81}, all of the connected components of $M_{+}$ are $M_{+, 1}$ and $M_{+, 2}$.
\end{proof}
\begin{proposition}
    \label{prop:52}
    In \ref{case:c-1}, then $M_{+}=M_{+, 1}$, which is connected.
    In \ref{case:c-2}, then $M_{+}=M_{+, 2}$, which is connected.
\end{proposition}
\begin{proof}
    We consider the case of \ref{case:c-1}.
    Since $s_{2}=l$, we have $E_{-}(P)\cap S^{2l-1}_{s}=\emptyset$.
    We suppose that $M_{+, 2}\neq\emptyset$ or $M_{+, 3}\neq\emptyset$ holds.
    By Lemma \ref{lemm:58}, if $M_{+, 2}\neq\emptyset$, then $E_{-}(P)\cap S^{2l-1}_{s}$ is non-empty, which is a contradiction.
    By Lemma \ref{lemm:60}, if $M_{+, 3}\neq\emptyset$, then there exists an element of $E_{-}(P)\cap S^{2l-1}_{s}$, which is a contradiction.
    Hence, $M_{+,2}=\emptyset$ and $M_{+, 3}=\emptyset$ hold.
    By Remark \ref{rema:53}, $M_{+}=M_{+, 1}$ holds.
    By Lemma \ref{lemm:58}, $M_{+}$ is path-connected, which is connected.

    The proof of the case of \ref{case:c-2} is similar to the above one of the case of \ref{case:c-1}.
\end{proof}
\begin{proposition}
    \label{prop:53}
    In \ref{case:d}, $M_{+}$ is connected.
    Then $M_{+, 1}$, $M_{+, 2}$ and $M_{+, 3}$ are non-empty.
\end{proposition}
\begin{proof}
    We show that $M_{+, 1}\neq \emptyset$, $M_{+, 2}\neq \emptyset$ and $M_{+, 3}\neq \emptyset$.
    By Remark \ref{rema:51}, there exists $x\in E_{\varepsilon}(P)\cap S^{2l-1}_{s}$ ($\varepsilon=1$ or $\varepsilon=-1$).
    By Lemma \ref{lemm:17-4} and \S\ref{sect:C-S}\ref{metric}, $\{P_{j}x\}_{j=1}^{m}$ is a pseudo-orthonormal system of $E_{-\varepsilon}(P)$.
    Then there exists $y\in E_{-\varepsilon}(P)\cap S^{2l-1}_{s}$ such that $\{P_{j}x\}_{j=1}^{m}\cup \{y\}$ is a pseudo-orthonormal system of $E_{-\varepsilon}(P)$.
    In fact, $m<l-s/2$ holds in \ref{case:d-1}; $m-r<l/2=l-s/2$ holds in \ref{case:d-2}; and $0<4\leq m\leqq \min{\{l-s_{1}, l-s_{2}\}}$ holds in \ref{case:d-3}.
    By Lemma \ref{lemm:56}, $x\in M_{+, (3-\varepsilon)/2}$ and $y\in M_{+, (3+\varepsilon)/2}$ hold.
    By Lemma \ref{lemm:54}, we can verify that $x/\sqrt{2}+y/\sqrt{2}\in M_{+, 3}$.
    
    We fix a point $z_{0}\in M_{+, 3}$.
    By Lemma \ref{lemm:60}, there exists a path $\alpha$ in $M_{+}$ from a point $x_{0}\in E_{+}(P)\cap S^{2l-1}_{s}$ to a point $y_{0}\in E_{-}(P)\cap S^{2l-1}_{s}$ through $z_{0}$.
    We take an arbitrary point $z\in M_{+}$.
    By Remark \ref{rema:53}, one of $z\in M_{+, 1}$, $z\in M_{+, 2}$ or $z\in M_{+, 3}$ holds.
    If $z\in M_{+, 1}$, by Lemma \ref{lemm:58}, there exists a path in $M_{+, 1}\subset M_{+}$ from $z$ to $x_{0}$.
    If $z\in M_{+, 2}$, by Lemma \ref{lemm:58}, there exists a path in $M_{+, 2}\subset M_{+}$ form $z$ to $y_{0}$.
    By using $\alpha$, there exists a path in $M_{+}$ from $z$ to $x_{0}$.
    If $z\in M_{+, 3}$, by Lemma \ref{lemm:60}, there exists a path in $M_{+}$ from $z$ to a point $x\in E_{+}(P)\cap S^{2l-1}_{s}$.
    By Lemmas \ref{lemm:57} and \ref{lemm:56}, there exists a path in $E_{+}(P)\cap S^{2l-1}_{s}\subset M_{+}$ from $x$ to $x_{0}$.
    Therefore, there exists a path in $M_{+}$ from $z$ to $x_{0}$, that is, $M_{+}$ is path-connected, which is connected.
\end{proof}
Propositions \ref{prop:51}, \ref{prop:52} and \ref{prop:53} reveal the connectedness of $M_{+}$.

\subsection{Inhomogeneity of the Focal Variety}
\label{sect:4-3}
By the method in \cite[5.8]{MR624227}, we obtain the following theorem.
\begin{theorem}
    \label{theo:inhomogeneous}
    For a Clifford system $\{P_{i}\}_{i=1}^{m}\ (P_{i}\in \mathrm{Sym}(\mathbb{R}^{2l}_{s}))$ of signature $(m, r)$ satisfying $m\equiv 0\pmod{4}$, $r\equiv 0\pmod{2}$ and $l>m$, each connected component of the focal variety $M_{+}$ is inhomogeneous.
\end{theorem}
\begin{remark}
    \label{rema:M-}
    In this paper, we do not investigate inhomogeneity of the focal variety $M_{-}$.
    Thus, we do not know whether a (connected) isoparametric hypersurface of OT-FKM-type whose focal variety is $M_{-}$ in the pseudo-sphere is inhomogeneous or not.
    (We refer to Theorem \ref{theo:inhomogeneous-2}.)
\end{remark}
We give a proof of Theorem \ref{theo:inhomogeneous} which is divided into some steps.
We use the settings in \S\ref{sect:4-1} and \S\ref{sect:4-2}.
Let $L\subset M_{+}$ be a connected component of $M_{+}$.
We note that $T_{x}L=T_{x}M_{+}$ and $T^{\bot}_{x}L=T^{\bot}_{x}M_{+}$ hold.
\begin{proposition}
    \label{prop:22}
    If $l>m$, there exist $x\in E_{+}(P)\setminus \{0\}$ and $y\in E_{-}(P)\setminus \{0\}$ such that $x+y\in L$.
\end{proposition}
\begin{proof}
    By Remark \ref{rema:51}; Lemmas \ref{lemm:56} and \ref{lemm:58}; and Propositions \ref{prop:51}, \ref{prop:52} and \ref{prop:53}, there exists $\varepsilon\in \{-1, 1\}$ such that $\emptyset \neq E_{\varepsilon}(P)\cap S^{2l-1}_{s}\subset L$ holds.
    We take an arbitrary $x'\in E_{\varepsilon}(P)\cap S^{2l-1}_{s}\subset L$.
    By Lemma \ref{lemm:17-4} and \S\ref{sect:C-S}\ref{metric}, $\{P_{j}x'\}_{j=1}^{m}$ is a pseudo-orthonormal system of $E_{-\varepsilon}(P)$.
    Since $l>m$, there exists $y'\in E_{-\varepsilon}(P)$ such that $\{P_{j}x'\}_{j=1}^{m}\cup \{y'\}$ is a pseudo-orthonormal system of $E_{-\varepsilon}(P)$.

    In \ref{case:a}, \ref{case:b} and \ref{case:c}, we can assume that $\langle y', y'\rangle=-1$.
    In fact, $0<l-m=s/2$ holds in \ref{case:a}; $r=m-l/2<l/2$ holds in \ref{case:b}; and $r=m<l$ holds in \ref{case:c}.
    In \ref{case:d}, we can assume that $\langle y', y'\rangle=1$.
    In fact, $m<l-s/2$ holds in \ref{case:d-1}; $m-r<l/2$ holds in \ref{case:d-2}; and $m\leqq l-s_{1}, l-s_{2}$ holds in \ref{case:d-3}.
    
    In \ref{case:a}, \ref{case:b} and \ref{case:c}, we assume that $\langle y', y'\rangle=-1$.
    Let 
    \begin{equation*}
        x:=\sqrt{2}x'
        \in E_{\varepsilon}(P)\setminus\{0\}
        ,\quad
        y:=y'
        \in E_{-\varepsilon}(P)\setminus\{0\}
        .
    \end{equation*}
    By Lemma \ref{lemm:54}, and Propositions \ref{prop:51} and \ref{prop:52}, we can verify that $x+y\in M_{+, (3-\varepsilon)/2}=L$ holds.

    In \ref{case:d}, we assume that $\langle y', y'\rangle=1$.
    Let 
    \begin{equation*}
        x:=\frac{x'}{\sqrt{2}}
        \in E_{\varepsilon}(P)\setminus\{0\}
        ,\quad
        y:=\frac{y'}{\sqrt{2}}
        \in E_{-\varepsilon}(P)\setminus\{0\}
        .
    \end{equation*}
    By Lemma \ref{lemm:54} and Proposition \ref{prop:53}, we can verify that $x+y\in M_{+, 3}\subsetneqq M_{+}=L$ holds.
\end{proof}
We define
\begin{equation*}
    N_{+}
    :=
    \left\{
        x\in L
        \ \middle|\ 
        \exists 
        v\in 
        \bigcap_{w\in T^{\bot}_{x}L}
        \ker{S_{w}}
        \ \mathrm{s.t}\ 
        \langle v, v\rangle>0
    \right\}
    \subset 
    L
    \subset 
    M_{+}
    .
\end{equation*}
In the following, we investigate $N_{+}$.
\begin{remark}
    \label{rema:Ntorikata}
    Let $\{Q_{i}\}_{i=1}^{m}$ be a pseudo-orthonormal basis of $\Sigma$ and $x\in L$.
    Then we have
    \begin{equation}
        \bigcap_{w\in T^{\bot}_{x}L}\ker{S_{w}}
        =
        \bigcap_{i=1}^{m}\ker{S_{Q_{i}x}}
        .
        \label{eq:43}
    \end{equation}
    In fact, let $v$ be an element of the right-hand side of \eqref{eq:43}.
    We fix $w\in T^{\bot}_{x}L$.
    By \S\ref{sect:F-V}, we can express
    \begin{equation*}
        w
        =
        \sum_{i=1}^{m}c_{(w)}^{i}Q_{i}x
        .
    \end{equation*}
    By the assumption of $v$, we have
    \begin{equation*}
        S_{w}(v)
        =
        \sum_{i=1}^{m}c_{(w)}^{i}
        S_{Q_{i}x}(v)
        =
        0
        .
    \end{equation*}
    Since $w$ is arbitrary, $v$ is an element of the left-hand side of \eqref{eq:43}.
    It is clear that the right-hand side of \eqref{eq:43} contains the left-hand side of \eqref{eq:43}.
    Therefore, for any pseudo-orthonormal basis $\{Q_{i}\}_{i=1}^{m}$ of $\Sigma$, we have
    \begin{equation*}
        N_{+}
        =
        \left\{
            x\in L
            \ \middle|\ 
            \exists 
            v\in 
            \bigcap_{i=1}^{m}
            \ker{S_{Q_{i}x}}
            \ \mathrm{s.t}\ 
            \langle v, v\rangle>0
        \right\}
        .
    \end{equation*}
\end{remark}
We set $R_{i}:=Q_{2i-1, 4}\ (i\in \{1, 2, \ldots, (m-2)/2\})$.
\begin{lemma}
    \label{lemm:19}
    We have $N_{+}\neq \emptyset$.
\end{lemma}
\begin{proof}
    We note that $R_{1}, R_{2}, \ldots, R_{(m-2)/2}$ are pairwise commutative.
    By Proposition \ref{prop:18} and using simultaneous eigenspace decomposition, we have the orthogonal decomposition
    \begin{equation*}
        \mathbb{R}^{2l}_{s}
        =
        \bigoplus_{(\varepsilon_{1}, \ldots, \varepsilon_{(m-2)/2})\in \{-1, 1\}^{(m-2)/2}}
        \left(
            \bigcap_{j=1}^{(m-2)/2}E_{\varepsilon_{j}}(R_{j})
        \right)
        .
    \end{equation*}
    By Remark \ref{rema:51}, there exists $(\varepsilon_{1}, \ldots, \varepsilon_{(m-2)/2})\in \{-1, 1\}^{(m-2)/2}$ such that 
    \begin{equation*} 
        \left(
            \bigcap_{j=1}^{(m-2)/2}E_{\varepsilon_{j}}(R_{j})
        \right)
        \cap S^{2l-1}_{s}
        \neq\emptyset
        .
    \end{equation*}
    We fix ${\displaystyle x\in \left(\bigcap_{j=1}^{(m-2)/2}E_{\varepsilon_{j}}(R_{j})\right)\cap S^{2l-1}_{s}}$.
    We note that we can express $P=R_{1}R_{3}\cdots R_{(m-6)/2}R_{(m-2)/2}$.
    Let $\epsilon:=\varepsilon_{1}\varepsilon_{3}\cdots \varepsilon_{(m-6)/2}\varepsilon_{(m-2)/2}\in \{-1, 1\}$.
    By Lemma \ref{lemm:17-3}, we can verify $x\in E_{\epsilon}(P)\cap S^{2l-1}_{s}$.
    By Lemma \ref{lemm:56}, $x\in M_{+, (3-\epsilon)/2}$ holds.
    We divide into the following cases:
    \begin{enumerate}[label=(\Alph*)]
        \item \label{case:A}
        \ref{case:a} and $x\in L$; \ref{case:b} and $x\in L$; \ref{case:c}; and \ref{case:d}.
        \item \label{case:B}
        \ref{case:a} and $x\notin L$; and \ref{case:b} and $x\notin L$.
    \end{enumerate}
    In \ref{case:A}, $x\in E_{\epsilon}(P)\cap S^{2l-1}_{s}\subset M_{+, (3-\epsilon)/2}\subset L$ holds.
    Let $v:=P_{1}P_{2}x$.
    We can verify that we express 
    \begin{equation*}
        v
        =
        \varepsilon(i, j)P_{i}P_{j}x
        ,\quad 
        (i, j)=(1, 2), (2, 1), (3, 4), (4, 3), \ldots, (m-1, m), (m, m-1)
        ,\quad
        \varepsilon(i, j)\in \{-1, 1\}
        .
    \end{equation*}
    Thus, by \S\ref{sect:F-V}\eqref{ker}, we have
    \begin{equation*}
        v\in 
        \bigcap_{i=1}^{m}
        \ker{S_{P_{i}x}}
        .
    \end{equation*}
    Since $r$ is even, $\eta_{11}=\eta_{22}$ holds.
    Thus, by \S\ref{sect:C-S}\ref{metric}, we have $\langle v, v\rangle=1>0$.
    By Remark \ref{rema:Ntorikata}, we have $x\in N_{+}$.

    We consider the case of \ref{case:B}.
    Then $x\notin E_{-\epsilon}(P)\cap S^{2l-1}_{s}\subset M_{+, (3+\epsilon)/2}=L$ holds.
    Let $x':=P_{m}x$.
    Since $r\neq m$, we can compute
    $
        \langle x', x'\rangle
        =
        1
    $
    by \S\ref{sect:C-S}\ref{metric}.
    By Lemma \ref{lemm:17-4}, we have $x'\in E_{-\epsilon}(P)\cap S^{2l-1}_{s}\subset M_{+, (3+\epsilon)/2}=L$.
    Let $v':=P_{1}P_{2}x'$.
    Similarly, we can verify
    \begin{equation*}
        v'\in 
        \bigcap_{i=1}^{m}
        \ker{S_{P_{i}x'}}
        ,
    \end{equation*}
    and $\langle v', v'\rangle=1$.
    By Remark \ref{rema:Ntorikata}, we have $x'\in N_{+}$.
\end{proof}
To prove Proposition \ref{prop:21}, we use the following Lemma \ref{lemm:21-2}.
\begin{lemma}
    \label{lemm:21-2}
    Let $\{Q_{i}\}_{i=1}^{m}$ be a pseudo-orthonormal basis of $\Sigma$.
    Then either
    $Q_{1}Q_{2}\cdots Q_{m}=P$
    or
    $Q_{1}Q_{2}\cdots Q_{m}=-P$
    holds.
\end{lemma}
\begin{proof}
    By \S\ref{sect:C-S}, there exists $A\in O(r, m-r)$ such that 
    \begin{equation*}
        \begin{bmatrix}
            Q_{1} & Q_{2} &\cdots &Q_{m}
        \end{bmatrix}
        =
        \begin{bmatrix}
            P_{1} & P_{2} &\cdots &P_{m}
        \end{bmatrix}
        A
        .
    \end{equation*}
    On the other hand, we can verify that any pseudo-orthogonal matrix is represented as a finite product of the following pseudo-orthogonal matrices:
    \begin{equation*}
        R_{a, a+1}(t)
        :=
        \begin{pmatrix}
            1&&&&&&&\\
            &\ddots&&&&&&\\
            &&1&&&&&\\
            &&&\cos{t}&-\sin{t}&&&\\
            &&&\sin{t}&\cos{t}&&&\\
            &&&&&1&&\\
            &&&&&&\ddots&\\
            &&&&&&&1
        \end{pmatrix}
        ,\quad t\in \mathbb{R},
    \end{equation*}
    \begin{equation*}
        S_{1, r+1}(t)
        :=
        \begin{pmatrix}
            \cosh{t}&&&&\sinh{t}&&&\\
            &1&&&&&&\\
            &&\ddots&&&&&\\
            &&&1&&&&&\\
            \sinh{t}&&&&\cosh{t}&&&\\
            &&&&&1&&&\\
            &&&&&&\ddots&\\
            &&&&&&&1
        \end{pmatrix}
        ,\quad t\in \mathbb{R},
    \end{equation*}
    \begin{equation*}
        T_{1}
        :=
        (-J_{r-1, 1})
        \oplus E_{m-r}
        ,\quad
        T_{2}
        :=
        E_{r}\oplus 
        (-J_{m-r-1, 1})
        ,
    \end{equation*}
    where $a\in \{1, \ldots, r-1\}\cup \{r+1, \ldots, m-1\}$; and where two $(\cos{t})$-s are located in $(a, a)$-entry and $(a+1, a+1)$-entry of $R_{a, a+1}(t)$; and where two $(\cosh{t})$-s are located in $(1, 1)$-entry and $(r+1, r+1)$-entry of $S_{1, r+1}(t)$.
    In fact, left-multiplying by finite elements of $\{R_{a, a+1}(t)\}_{a, t}\cup\{S_{1, r+1}(t)\}_{t}$ changes any pseudo-orthogonal matrix into one of $E_{m}$, $T_{1}$, $T_{2}$ or $T_{1}T_{2}$.
    
    Thus, it suffices to show that $Q_{1}Q_{2}\cdots Q_{m}$ equals $P$ or $-P$ in the case where $A$ is $R_{a, a+1}(t)$, $S_{1, r+1}(t)$, $T_{1}$ and $T_{2}$.
    We consider the case where $A=R_{a, a+1}(t)$.
    We have $\eta_{aa}=\eta_{a+1, a+1}$ since $a\neq r$.
    Thus, we can verify that $Q_{1}\cdots Q_{m}=P$ holds by direct computation and \S \ref{sect:C-S}\ref{change-PQ}.
    In the case where $A=S_{1, r+1}$, we can verify that $Q_{1}\cdots Q_{m}=P$ holds similarly.
    In the case where $A$ is $T_{1}$ or $T_{2}$, $Q_{1}\cdots Q_{m}=-P$ holds.
\end{proof}
\begin{proposition}
    \label{prop:21}
    We have $N_{+}\subset E_{+}(P)\cup E_{-}(P)$.
\end{proposition}
\begin{proof}
    We take an arbitrary $x\in N_{+}$.
    Then there exists 
    \begin{equation*}
        v
        \in 
        \bigcap_{i=1}^{m}
        \ker{S_{P_{i}x}}
        ,
    \end{equation*}
    such that $\langle v, v\rangle=1>0$.
    By Remark \ref{rema:Ntorikata}, we can express
    \begin{equation*}
        v
        =
        Q_{1}P_{1}x
        =
        Q_{2}P_{2}x
        =
        \cdots
        =
        Q_{m}P_{m}x
        ,
    \end{equation*}
    where, for any $i\in \{1, \ldots, m\}$, there exists $Q_{i}\in \Sigma$ such that $\langle P_{i}, Q_{i}\rangle=0$.
    Since we have
    \begin{equation*}
        1
        =
        \langle v, v\rangle
        =
        \langle Q_{1}P_{1}x, Q_{1}P_{1}x\rangle
        =
        \langle Q_{1}, Q_{1}\rangle
        \langle P_{1}, P_{1}\rangle
        \langle x, x\rangle
        =
        \eta_{11}
        \langle Q_{1}, Q_{1}\rangle
        ,
    \end{equation*}
    by \S\ref{sect:C-S}\ref{metric}, $\langle Q_{1}, Q_{1}\rangle=\eta_{11}$ holds.
    Since $r$ is even, $\eta_{11}=\eta_{22}$ holds.
    We can take $\omega_{1}\in \{2, 3, \ldots, m\}$ satisfying $\langle Q_{1}, P_{\omega_{1}}\rangle\neq 0$.
    Then $P_{1}, Q_{1}, \{P_{i}\}_{i\in \{2, 3, \ldots, m\}\setminus\{\omega_{1}\}}$ is a basis of $\Sigma$.
    By using the generalized Gram-Schmidt process, we can change $P_{1}, Q_{1}, \{P_{i}\}_{i\in \{2, 3, \ldots, m\}\setminus\{\omega_{1}\}}$ into a pseudo-orthonormal basis $\{P^{(1)}_{i}\}_{i=1}^{m}$ of $\Sigma$ satisfying $P^{(1)}_{1}=P_{1}, P^{(1)}_{2}=Q_{1}$.
    (We refer to \cite[Chapter 1, PROPOSITION 1.2, PROOF]{MR1383318} to understand the generalized Gram-Schmidt process.)
    By Remark \ref{rema:Ntorikata}, we can express
    \begin{equation*}
        v
        =
        P^{(1)}_{2}P^{(1)}_{1}x
        =
        -P^{(1)}_{1}P^{(1)}_{2}x
        =
        Q^{(1)}_{3}P^{(1)}_{3}x
        =
        Q^{(1)}_{4}P^{(1)}_{4}x
        =
        \cdots
        =
        Q^{(1)}_{m}P^{(1)}_{m}x
        ,
    \end{equation*}
    where, for any $i\in \{3, \ldots, m\}$, there exists $Q^{(1)}_{i}\in \Sigma$ such that $\langle P^{(1)}_{i}, Q^{(1)}_{i}\rangle=0$.

    We fix $b\in \{1, 2, \ldots, (m-2)/2\}$.
    We assume that there exists a pseudo-orthonormal basis $\{P^{(b)}_{i}\}_{i=1}^{m}$ of $\Sigma$ such that we can express
    \begin{align*}
        v
        &=
        P^{(b)}_{2}P^{(b)}_{1}x
        =
        -
        P^{(b)}_{1}P^{(b)}_{2}x
        =
        \cdots
        =
        P^{(b)}_{2b}P^{(b)}_{2b-1}x
        =
        -
        P^{(b)}_{2b-1}P^{(b)}_{2b}x\\
        &=
        Q^{(b)}_{2b+1}P^{(b)}_{2b+1}x
        =
        Q^{(b)}_{2b+2}P^{(b)}_{2b+2}x
        =
        \cdots
        =
        Q^{(b)}_{m}P^{(b)}_{m}x
        ,
    \end{align*}
    where, for any $i\in \{2b+1, 2b+2, \ldots, m\}$, there exists $Q^{(b)}_{i}\in \Sigma$ such that $\langle P^{(b)}_{i}, Q^{(b)}_{i}\rangle=0$.
    For $a\in \{1, 2, \ldots, 2b\}$, by \S\ref{sect:C-S}\ref{metric}, we have
    \begin{align*}
        0
        &=
        (-1)^{a+1}
        \langle P^{(b)}_{2b+1}, P^{(b)}_{a+(-1)^{a+1}}\rangle
        \langle P^{(b)}_{a}, P^{(b)}_{a}\rangle
        \langle x, x\rangle
        =
        (-1)^{a+1}
        \langle P^{(b)}_{2b+1}P^{(b)}_{a}x, P^{(b)}_{a+(-1)^{a+1}}P^{(b)}_{a}x\rangle\\
        &=
        \langle P^{(b)}_{2b+1}P^{(b)}_{a}x, Q^{(b)}_{2b+1}P^{(b)}_{2b+1}x\rangle
        =
        -
        \langle P^{(b)}_{2b+1}P^{(b)}_{a}x, P^{(b)}_{2b+1}Q^{(b)}_{2b+1}x\rangle\\
        &=
        -
        \langle P^{(b)}_{2b+1}, P^{(b)}_{2b+1}\rangle
        \langle P^{(b)}_{a}, Q^{(b)}_{2b+1}\rangle
        \langle x, x\rangle
        =
        -
        \eta_{2b+1, 2b+1}\langle P^{(b)}_{a}, Q^{(b)}_{2b+1}\rangle
        .
    \end{align*}
    Thus, we have $\langle P^{(b)}_{a}, Q^{(b)}_{2b+1}\rangle=0$ for any $a\in \{1, 2, \ldots, 2b, 2b+1\}$.
    Since we have
    \begin{align*}
        1
        =
        \langle v, v\rangle
        =
        \langle Q^{(b)}_{2b+1}P^{(b)}_{2b+1}x, Q^{(b)}_{2b+1}P^{(b)}_{2b+1}x\rangle
        =
        \eta_{2b+1,2b+1}
        \langle Q^{(b)}_{2b+1}, Q^{(b)}_{2b+1}\rangle
        ,
    \end{align*}
    by \S\ref{sect:C-S}\ref{metric}, $\langle Q^{(b)}_{2b+1}, Q^{(b)}_{2b+1}\rangle=\eta_{2b+1, 2b+1}$ holds.
    Since $r$ is even, $\eta_{2b+1, 2b+1}=\eta_{2b+2, 2b+2}$.
    We can take $\omega_{b}\in \{2b+2, 2b+3, \ldots, m\}$ satisfying $\langle Q^{(b)}_{2b+1}, P_{\omega_{b}}\rangle\neq 0$.
    Then $P^{(b)}_{1}, \ldots, P^{(b)}_{2b}, P^{(b)}_{2b+1}, Q^{(b)}_{2b+1}$,\! $\{P_{i}\}_{i\in \{2b+2, 2b+3, \ldots, m\}\setminus \{\omega_{b}\}}$ is a basis of $\Sigma$.
    By using the generalized Gram-Schmidt process, we can change $P^{(b)}_{1}, \ldots, P^{(b)}_{2b}, P^{(b)}_{2b+1}, Q^{(b)}_{2b+1}$,\! $\{P_{i}\}_{i\in \{2b+2, 2b+3, \ldots, m\}\setminus \{\omega_{b}\}}$ into a pseudo-orthonormal basis $\{P^{(b+1)}_{i}\}_{i=1}^{m}$ of $\Sigma$ such that $P^{(b+1)}_{i}=P^{(b)}_{i}$ for any $(i\in \{1, 2, \ldots, 2b+1\})$ and $P^{(b+1)}_{2b+2}=Q^{(b)}_{2b+1}$ hold.
    By Remark \ref{rema:Ntorikata}, we can express
    \begin{align*}
        v
        &=
        P^{(b+1)}_{2}P^{(b+1)}_{1}x
        =
        -
        P^{(b+1)}_{1}P^{(b+1)}_{2}x
        =
        \cdots
        =
        P^{(b+1)}_{2b+2}P^{(b+1)}_{2b+1}x
        =
        -
        P^{(b+1)}_{2b+1}P^{(b+1)}_{2b+2}x\\
        &=
        Q^{(b+1)}_{2b+3}P^{(b+1)}_{2b+3}x
        =
        Q^{(b+1)}_{2b+4}P^{(b+1)}_{2b+4}x
        =
        \cdots
        =
        Q^{(b+1)}_{m}P^{(b+1)}_{m}x
        ,
    \end{align*}
    where, for any $i\in \{2b+3, 2b+4, \ldots, m\}$, there exists $Q^{(b+1)}_{i}\in \Sigma$ such that $\langle P^{(b+1)}_{i}, Q^{(b+1)}_{i}\rangle=0$. 

    By repetition of the above procedure for $b=1, 2, \ldots, (m-2)/2$, we obtain a pseudo-orthonormal basis $\{P^{(m/2)}_{i}\}_{i=1}^{m}$ of $\Sigma$ and we can express
    \begin{align}
        v
        &=
        P^{(m/2)}_{2}P^{(m/2)}_{1}x
        =
        -
        P^{(m/2)}_{1}P^{(m/2)}_{2}x
        =
        P^{(m/2)}_{4}P^{(m/2)}_{3}x
        =
        -
        P^{(m/2)}_{3}P^{(m/2)}_{4}x\notag\\
        &=
        \cdots
        =
        P^{(m/2)}_{m}P^{(m/2)}_{m-1}x
        =
        -
        P^{(m/2)}_{m-1}P^{(m/2)}_{m}x
        .\label{eq:11}
    \end{align}
    By \eqref{eq:11}, we can verify $P^{(m/2)}_{1}P^{(m/2)}_{2}\cdots P^{(m/2)}_{m-1}P^{(m/2)}_{m}x=(-1)^{m/4}x$.
    Thus, we have
    \begin{equation*}
        x\in E_{+}(P^{(m/2)}_{1}P^{(m/2)}_{2}\cdots P^{(m/2)}_{m})\cup E_{-}(P^{(m/2)}_{1}P^{(m/2)}_{2}\cdots P^{(m/2)}_{m})
        .
    \end{equation*}
    We note that $E_{+}(-P)=E_{-}(P)$ and $E_{-}(-P)=E_{+}(P)$.
    By Lemma \ref{lemm:21-2}, we have $x\in E_{+}(P)\cup E_{-}(P)$.
    Since $x$ is arbitrary, we have $N_{+}\subset E_{+}(P)\cup E_{-}(P)$.
\end{proof}
If we suppose that $L$ is homogeneous, then we find that the subset $N_{+}\subset L$ covers $L$.
\begin{proposition}
    \label{prop:20}
    If $L$ is homogeneous, then we have $N_{+}=L$.
\end{proposition}
\begin{proof}
    We show that $L\subset N_{+}$.
    We suppose that $L$ is homogeneous, that is, there exists a subgroup $K$ of $O(s, 2l-s)$ which acts on $L$ transitively.
    By Lemma \ref{lemm:19}, we can take $x\in N_{+}$.
    By Remark \ref{rema:Ntorikata}, there exists 
    \begin{equation*}
        v
        \in 
        \bigcap_{i=1}^{m}
        \ker{S_{P_{i}x}}
        ,
    \end{equation*}
    such that $\langle v, v\rangle>0$.
    We take an arbitrary $k\in K$.
    Since $k$ is an isometry, $\{(dk)_{x}(P_{i}x)\}_{i=1}^{m}$ is a pseudo-orthonormal basis of $T^{\bot}_{k(x)}L$.
    Thus, there exists a pseudo-orthonormal basis $\{Q_{i}\}_{i=1}^{m}$ of $\Sigma$ such that $(dk)_{x}(P_{i}x)=Q_{i}k(x)$ for all $i\in \{1, \ldots, m\}$.
    We note that the shape operator is invariant under $K$.
    By the assumption of $v$, we have
    \begin{equation*}
        S_{Q_{i}k(x)}(dk)_{x}(v)
        =
        S_{(dk)_{x}(P_{i}x)}(dk)_{x}(v)
        =
        (dk)_{x}(S_{P_{i}x}v)
        =
        0
        ,\quad
        i\in \{1, \ldots, m\}
        .
    \end{equation*}
    Hence, we have $\langle (dk)_{x}(v), (dk)_{x}(v)\rangle>0$ and 
    \begin{equation*}
        (dk)_{x}(v)
        \in 
        \bigcap_{i=1}^{m}
        \ker{S_{Q_{i}k(x)}}
        .
    \end{equation*}
    By Remark \ref{rema:Ntorikata}, $k(x)\in N_{+}$ holds.
    Since $k$ is arbitrary, we have $L\subset N_{+}$.
\end{proof}
Therefore, we can find that $L$ is inhomogeneous.
\begin{proof}
    (The Proof of Theorem \ref{theo:inhomogeneous})
    By Proposition \ref{prop:22}, there exist $x\in E_{+}(P)\setminus \{0\}$ and $y\in E_{-}(P)\setminus \{0\}$ such that $x+y\in L$.
    Now, $x+y\notin E_{+}(P)\cup E_{-}(P)$ holds.
    By Proposition \ref{prop:21}, $x+y\notin N_{+}$ holds.
    This means that $N_{+}\subsetneqq L$.
    By the contrapositive of Proposition \ref{prop:20}, $L$ is inhomogeneous.
\end{proof}

\subsection{Inhomogeneity of Isoparametric Hypersurfaces}
Eventually, we have the following main theorem.
\begin{theorem}
    \label{theo:inhomogeneous-2}
    For a Clifford system $\{P_{i}\}_{i=1}^{m}\ (P_{i}\in \mathrm{Sym}(\mathbb{R}^{2l}_{s}))$ of signature $(m, r)$ satisfying $m\equiv 0\pmod{4}$, $r\equiv 0\pmod{2}$ and $l>m$, each connected component of $M_{c}\ (c\in \mathrm{W}_{\mathrm{RN}}(f)\cap (-1, \infty))$ is inhomogeneous.
\end{theorem}
We give a proof of Theorem \ref{theo:inhomogeneous-2} which is divided into some steps.
We use the settings in \S\ref{sect:4-1}, \S\ref{sect:4-2} and \S\ref{sect:4-3}.
We fix $c\in \mathrm{W}_{\mathrm{RN}}(f)\cap (-1, \infty)$.
Let $\delta$ be the type of $M_{c}$.
Let 
$
    \bot L(\delta)
    :=
    \left\{
        (x, v)
        \ \middle|\
        x\in L,\ v\in T^{\bot}_{x}L,\ \langle v, v\rangle=\delta
    \right\}
$
and $M_{c, L}:=\varphi_{t_{c}}(\bot L(\delta))\subset M_{c}$.
\begin{lemma}
    \label{lemm:64}
    The set $M_{c, L}$ is a connected component of $M_{c}$.
\end{lemma}
\begin{proof}
    In \ref{case:c} and \ref{case:d}, $L=M_{+}$ holds.
    Thus, $M_{c, L}=M_{c}$ holds.

    We consider the cases of \ref{case:a} and \ref{case:b}.
    By Proposition \ref{prop:51}, all of the connected components of $M_{+}$ are $M_{+, 1}$ and $M_{+, 2}$.
    Here, we note that $S^{m-1}_{r}$ and $H^{m-1}_{r-1}$ are path-connected since $r\in \{0, 2, \ldots, m-2, m\}$.
    Since the base spaces $M_{+, 1}, M_{+, 2}$ and the fiber $S^{m-1}_{r}$\ ($\delta=1$) or $H^{m-1}_{r-1}$\ ($\delta=-1$) are path-connected, we can verify that $\bot M_{+, 1}(\delta)$ and $\bot M_{+, 2}(\delta)$ are path-connected, which are connected.
    Since $\varphi_{t_{c}}$ is a diffeomorphism, $\varphi_{t_{c}}(\bot M_{+, 1}(\delta))$ and $\varphi_{t_{c}}(\bot M_{+, 2}(\delta))$ are connected components of $M_{c}$.
\end{proof}
\begin{lemma}
    \label{prop:4}
    We assume that there exists a connected Lie subgroup $K_{0}\subset O(s, 2l-s)$ such that $K_{0}$ acts on $M_{c, L}$ transitively.
    We take arbitrary $x', y'\in M_{c, L}\subset S^{2l-1}_{s}$.
    We also take an arbitrary $k\in K_{0}$ such that $k(x')=y'$.
    There exist unique pairs $(x, v), (y, u)\in \bot L(\delta)$ such that $\varphi_{t_{c}}(x, v)=x'$ and $\varphi_{t_{c}}(y, u)=y'$.
    Then $k(x)=y$ holds.
\end{lemma}
\begin{proof}
    Since $K_{0}$ is path-connected, there exists a path $\rho \colon [0, 1]\to K_{0}$ such that $\rho(0)=e,\ \rho(1)=k$, where $e$ is the identity of $K_{0}$.
    The curve $\psi_{s}\colon \mathbb{R}\to S^{2l-1}_{s}$ is given by $\psi_{s}=\rho(s)(\gamma_{x, v}(t))$ for $s\in [0, 1]$.
    We fix $s\in [0, 1]$.
    We have $\psi_{s}(0)=\rho(s)(x)$ and $\psi_{s}'(0)=(d\rho(s))_{x}(\gamma_{x, v}'(0))=(d\rho(s))_{x}(v)$.
    Since $\rho(s)$ is an isometry, $\psi_{s}=\gamma_{\rho(s)(x), (d\rho(s))_{x}(v)}$ holds.
    By \S\ref{sect:F-V}\eqref{eq:tvnv} and \S\ref{sect:I-H}\ref{tamotsu}, we have $\psi_{s}(t_{c})=\rho(s)(x')\in M_{c, L}$ and $\psi_{s}'(t_{c})=(d\rho(s))_{x'}(-\xi_{x'})\in T^{\bot}_{\rho(s)(x')}M_{c, L}$.
    We note that the codimension of $M_{c, L}\subset S^{2l-1}_{s}$ is $1$.
    Since $\psi_{0}'(t_{c})=-\xi_{x'}$ holds and $\rho$ is continuous, we have $\psi_{s}'(t_{c})=-\xi_{\rho(s)(x')}$.
    For $s=1$, we have $\psi_{1}(t_{c})=k(x')=y'=\gamma_{y, u}(t_{c})$ and $\psi_{1}'(t_{c})=\gamma_{y, u}'(t_{c})$.
    By uniqueness of geodesics, we have $\psi_{1}=\gamma_{y, u}$.
    Therefore, we have $k(x)=\psi_{1}(0)=\gamma_{y, u}(0)=y$.
\end{proof}
\begin{proposition}
    \label{prop:5}
    If $M_{c, L}$ is homogeneous, then $L$ is homogeneous.
\end{proposition}
\begin{proof}
    We assume that $M_{c, L}$ is homogeneous, that is, there exists a Lie subgroup $K\subset O(s, 2l-s)$ such that $K$ acts on $M_{c, L}$ transitively. 
    Since $M_{c, L}$ is connected, the identity component $K_{0}\subset K$ also acts on $M_{c, L}$ transitively by \S\ref{sect:I-H}\ref{taniseibun}.
    
    We show that $K_{0}$ acts on $L$ transitively.
    We take arbitrary $k\in K_{0},\ x\in L$.
    We take an arbitrary $v\in (\bot L(\delta))_{x}$.
    Let $x':=\varphi_{t_{c}}(x, v)$.
    Since $K_{0}$ acts on $M_{c, L}$ transitively, $y':=k(x')\in M_{c, L}$.
    Since $\varphi_{t_{c}}$ is bijection, there exists a unique pair $(y, u)\in \bot L(\delta)$ such that $y'=\varphi_{t_{c}}(y, u)$.
    By Lemma \ref{prop:4}, we have $k(x)=y\in L$.
    The mapping $\mathcal{L}\colon K_{0}\times L\to L$ is given by $\mathcal{L}(k, x):=k(x)$.
    Since $L$ is $K_{0}$-invariant, $\mathcal{L}$ is a smooth action.
    We show that $\mathcal{L}$ is transitive.
    It suffices to show that $L\subset \mathcal{L}(K_{0}, x)$ for a point $x\in L$.
    We fix $x\in L$.
    We take an arbitrary $v\in (\bot L(\delta))_{x}$.
    Let $x':=\varphi_{t_{c}}(x, v)$.
    We take an arbitrary $y\in L$.
    We take an arbitrary $u\in (\bot L(\delta))_{y}$.
    Let $y':=\varphi_{t_{c}}(y, u)$.
    Since $K_{0}$ acts on $M_{c, L}$ transitively, there exists $k\in K_{0}$ such that $k(x')=y'$.
    By Lemma \ref{prop:4}, we have $y=k(x)\in \mathcal{L}(K_{0}, x)$.
    Since $y$ is arbitrary, $L\subset \mathcal{L}(K_{0}, x)$ holds.

    Therefore, $\mathcal{L}$ is a transitive group action of $K_{0}$ on $L$, that is, $L$ is homogeneous.
\end{proof}
\begin{proof}
    (The Proof of Theorem \ref{theo:inhomogeneous-2})
    By Theorem \ref{theo:inhomogeneous}, each connected component $L\subset M_{+}$ is inhomogeneous.
    By the contrapositive of Proposition \ref{prop:5}, $M_{c, L}$ is inhomogeneous.
\end{proof}

\section{Some Examples}
By Theorems \ref{theo:C-exist} and \ref{theo:inhomogeneous-2}, we obtain a lot of examples of inhomogeneous isoparametric hypersurfaces of OT-FKM-type.
Since we can take an arbitrary positive integer $d$ in Proposition \ref{prop:9}, it is not difficult to satisfy the condition $l>m$ which is the assumption in Theorems \ref{theo:inhomogeneous} and \ref{theo:inhomogeneous-2}.
In addition, a Clifford system of signature $(m, r)$ on $\mathbb{R}^{2l}_{s}$ satisfying $m\equiv 0\pmod{4}$, $r\equiv 0\pmod{2}$ and $d=1$ which is given by Propositions \ref{prop:8} and \ref{prop:9} satisfies $l>m$ except for $(m, r)=(4, 2)$.

The idea of the diffeomorphism type of \S\ref{sect:Example---1} and \S\ref{sect:Example---2} is based on \cite[Theorem 1.2, Theorem 1.3]{MR4599604}.

\subsection{Example---1}
\label{sect:Example---1}
Let $\{P_{i}\}_{i=1}^{4}$ be the Clifford system of signature $(4, 0)$ which is made from $A^{(2, 2)}_{1}$ and $A^{(2, 2)}_{2}$ in Lemma \ref{lemm:15} by using Lemma \ref{lemm:14} and Proposition \ref{prop:9}.
Since $s=l=8$, $d=1$ and $(m, r)=(4, 0)$, $\{P_{i}\}$ is in \ref{case:a}.
By Proposition \ref{prop:51}, all of connected components of $M_{+}$ are $M_{+, 1}$ and $M_{+, 2}$.
By \eqref{rema:63}, $\mathrm{W}_{\mathrm{RN}}(f)\cap (-1, \infty)=(-1, 1)$.
By Theorem \ref{theo:inhomogeneous-2}, the subsets $M_{c, M_{+, 1}},\ M_{c, M_{+, 2}}\subset S^{2l-1}_{s}$ are inhomogeneous isoparametric hypersurfaces of OT-FKM-type.

Let $h:=\left.P_{4}\right|_{M_{+, 1}}$.
Then $h(M_{+, 1})=M_{+, 2}$ holds and $h\colon M_{+, 1}\to M_{+, 2}$ is a diffeomorphism.
Thus, $\bot M_{+, 1}(1)\cong \bot M_{+, 2}(1)$ holds since $M_{+, 1}\cong M_{+, 2}$ holds.
Hence, $M_{c_{1}, M_{+, 1}}\cong M_{c_{2}, M_{+, 2}}$ holds for $c_{1}, c_{2}\in (-1, 1)$

Let
$
    (S^{4}_{4})_{+}
    :=
    \left\{
        (y^{1}, y^{2}, y^{3}, y^{4}, y^{5})
        \in S^{4}_{4}
        \ \middle|\ 
        y^{5}\geqq 1
    \right\}
    .
$
Then $M_{+, 1}$ is diffeomorphic to $S^{7}_{4}\times (S^{4}_{4})_{+}$.
In fact, we can make the diffeomorphism from $M_{+, 1}$ to $S^{7}_{4}\times (S^{4}_{4})_{+}$.
Let $P:=P_{1}P_{2}P_{3}P_{4}$.
Let 
\begin{align*}
    \bm{a}_{1}
    &:=
    \frac{\bm{e}_{1}+\bm{e}_{8}}{\sqrt{2}},
    &
    \bm{a}_{2}
    &:=
    \frac{\bm{e}_{2}+\bm{e}_{7}}{\sqrt{2}},
    &
    \bm{a}_{3}
    &:=
    \frac{\bm{e}_{3}-\bm{e}_{6}}{\sqrt{2}},
    &
    \bm{a}_{4}
    &:=
    \frac{\bm{e}_{4}-\bm{e}_{5}}{\sqrt{2}},\\
    \bm{a}_{5}
    &:=
    \frac{\bm{e}_{9}+\bm{e}_{16}}{\sqrt{2}},
    &
    \bm{a}_{6}
    &:=
    \frac{\bm{e}_{10}+\bm{e}_{15}}{\sqrt{2}},
    &
    \bm{a}_{7}
    &:=
    \frac{\bm{e}_{11}-\bm{e}_{14}}{\sqrt{2}},
    &
    \bm{a}_{8}
    &:=
    \frac{\bm{e}_{12}-\bm{e}_{13}}{\sqrt{2}}
    .
\end{align*}
Then $\{\bm{a}_{i}\}_{i=1}^{8}$ is a pseudo-orthonormal basis of $E_{+}(P)$.
\begin{align*}
    \bm{b}_{1}
    &:=
    \frac{\bm{e}_{1}-\bm{e}_{8}}{\sqrt{2}},
    &
    \bm{b}_{2}
    &:=
    \frac{\bm{e}_{2}-\bm{e}_{7}}{\sqrt{2}},
    &
    \bm{b}_{3}
    &:=
    \frac{\bm{e}_{3}+\bm{e}_{6}}{\sqrt{2}},
    &
    \bm{b}_{4}
    &:=
    \frac{\bm{e}_{4}+\bm{e}_{5}}{\sqrt{2}},\\
    \bm{b}_{5}
    &:=
    \frac{\bm{e}_{9}-\bm{e}_{16}}{\sqrt{2}},
    &
    \bm{b}_{6}
    &:=
    \frac{\bm{e}_{10}-\bm{e}_{15}}{\sqrt{2}},
    &
    \bm{b}_{7}
    &:=
    \frac{\bm{e}_{11}+\bm{e}_{14}}{\sqrt{2}},
    &
    \bm{b}_{8}
    &:=
    \frac{\bm{e}_{12}+\bm{e}_{13}}{\sqrt{2}}
    .
\end{align*}
Then $\{\bm{b}_{j}\}_{j=1}^{8}$ is a pseudo-orthonormal basis of $E_{-}(P)$.
For any $z\in M_{+, 1}$, we can express
\begin{equation*}
    z_{+}
    :=
    \sum_{i=1}^{8}c^{i}\bm{a}_{i}
    ,\quad
    z_{-}
    :=
    \sum_{j=1}^{8}d^{j}\bm{b}_{j}
    ,\quad
    z
    =
    z_{+}+z_{-}
    ,\quad
    -\sum_{i=1}^{4}(c^{i})^{2}+\sum_{i=5}^{8}(c^{i})^{2}\geqq 1
    .
\end{equation*}
We note that the condition $\langle P_{i}z_{+}, z_{-}\rangle=0\ (i\in \{1, 2, 3, 4\})$ can be expressed by using $\{c^{i}\}_{i=1}^{8}, \{d^{j}\}_{j=1}^{8}$.
Let $\bm{c}=\bm{c}(z):=(c^{1}, c^{2}, \ldots, c^{8})$ and 
\begin{equation*}
    (A^{i}_{j}(\bm{c}))_{1\leqq i, j\leqq 4}
    :=
    \begin{pmatrix}
        c^{7}&-c^{8}&-c^{5}&c^{6}\\
        -c^{6}&c^{5}&-c^{8}&c^{7}\\
        c^{5}&c^{6}&c^{7}&c^{8}\\
        -c^{8}&-c^{7}&c^{6}&c^{5}
    \end{pmatrix}
    ^{-1}
    \begin{pmatrix}
        -c^{3}&c^{4}&c^{1}&-c^{2}\\
        c^{2}&-c^{1}&c^{4}&-c^{3}\\
        -c^{1}&-c^{2}&-c^{3}&-c^{4}\\
        c^{4}&c^{3}&-c^{2}&-c^{1}
    \end{pmatrix}
    ,
\end{equation*}
\begin{align*}
    \bm{q}_{i}(\bm{c})
    &:=
    \frac{1}{\sqrt{1-{\displaystyle \sum_{j=1}^{4}(A^{j}_{i}(\bm{c}))^{2}}}}
    \begin{bmatrix}
        \bm{b}_{1}&\bm{b}_{2}&\bm{b}_{3}&\bm{b}_{4}&\bm{b}_{5}&\bm{b}_{6}&\bm{b}_{7}&\bm{b}_{8}
    \end{bmatrix}
    \begin{pmatrix}
        \delta_{i1}\\
        \delta_{i2}\\
        \delta_{i3}\\
        \delta_{i4}\\
        -A^{1}_{i}(\bm{c})\\
        -A^{2}_{i}(\bm{c})\\
        -A^{3}_{i}(\bm{c})\\
        -A^{4}_{i}(\bm{c})
    \end{pmatrix}
    ,\quad
    i\in \{1, 2, 3, 4\}
    .
\end{align*}
We can verify that $\bm{q}_{i}(\bm{c})\in E_{-}(P)\ (i\in \{1, 2, 3, 4\})$, $\langle \bm{q}_{i}(\bm{c}), \bm{q}_{j}(\bm{c})\rangle=-\delta_{ij}\ (i, j\in \{1, 2, 3, 4\})$ and
\begin{equation*}
    \left\{
        w\in E_{-}(P)
        \ \middle|\ 
        \langle P_{i}z_{+}, w\rangle=0
        \ (\forall i\in \{1, 2, 3, 4\})
    \right\}
    =
    \mathrm{span}
    \left\{
        \bm{q}_{i}
        \ \middle|\ 
        i\in \{1, 2, 3, 4\}
    \right\}
    .
\end{equation*}
Here, let
\begin{align*}
    y^{5}(z)
    &:=
    \sqrt{
        -\sum_{i=1}^{4}(c^{i})^{2}+\sum_{i=5}^{8}(c^{i})^{2}
    }
    ,&
    \bm{x}(z)
    &:=
    \frac{\bm{c}(z)}{y^{5}(z)}
    ,\\
    y^{j}(z)
    &:=
    -\langle z_{-}, \bm{q}_{j}(\bm{c}(z))\rangle
    \ (j\in \{1, 2, 3, 4\})
    ,&
    \bm{y}(z)
    &=
    (y^{1}(z), y^{2}(z), y^{3}(z), y^{4}(z), y^{5}(z))
    .
\end{align*}
We can verify that $\bm{x}(z)\in S^{7}_{4}$ and $\bm{y}(z)\in (S^{4}_{4})_{+}$ hold, and the mapping $M_{+, 1}\ni z\mapsto (\bm{x}(z), \bm{y}(z))\in S^{7}_{4}\times (S^{4}_{4})_{+}$ is a diffeomorphism.

Let $(S^{4}_{4})_{-}:=S^{4}_{4}\setminus (S^{4}_{4})_{+}.$
Since $(S^{4}_{4})_{+}\cong (S^{4}_{4})_{-}$ holds, $M_{+, 2}\cong S^{7}_{4}\times (S^{4}_{4})_{-}$ holds.
Thus, $M_{+}\cong S^{7}_{4}\times S^{4}_{4}$ holds.
Hence, $M_{c, M_{+, 1}}\cong S^{7}_{4}\times (S^{4}_{4})_{+}\times S^{3}_{0}$ and $M_{c, M_{+, 2}}\cong S^{7}_{4}\times (S^{4}_{4})_{-}\times S^{3}_{0}$ hold.
Therefore, $M_{c}\cong S^{7}_{4}\times S^{4}_{4}\times S^{3}_{0}$ holds.

\subsection{Example---2}
\label{sect:Example---2}
Let $\{P_{i}\}_{i=1}^{4}$ be the Clifford system of signature $(4, 4)$ which is made from $A^{(2, 0)}_{1}$ and $A^{(2, 0)}_{2}$ in Lemma \ref{lemm:15} by using Lemma \ref{lemm:13} and Proposition \ref{prop:9}.
Since $s=l=8$, $d=1$, $(m, r)=(4, 4)$ and $(s_{1}, s_{2})=(4, 4)$, $\{P_{i}\}$ is in \ref{case:d-3}.
By Proposition \ref{prop:53}, $M_{+}$ is connected.
By \eqref{rema:63}, $\mathrm{W}_{\mathrm{RN}}(f)\cap (-1, \infty)=(1, \infty)$.
By Theorem \ref{theo:inhomogeneous-2}, $M_{c}$ is an inhomogeneous isoparametric hypersurface of OT-FKM-type.

We can show that $M_{+, 1}=E_{+}(P)\cap S^{15}_{8}$, $M_{+, 2}=E_{-}(P)\cap S^{15}_{8}$, $M_{+}\cong S^{7}_{4}\times S^{4}_{0}$ and $M_{c}\cong S^{7}_{4}\times S^{4}_{0}\times H^{3}_{3}$ hold.
The proof is similar to \S \ref{sect:Example---1}.

\bibliographystyle{abbrv}
\bibliography{article1}

\end{document}